\documentclass[11pt,a4paper]{article}
\usepackage[latin1]{inputenc}
\usepackage[T1]{fontenc}
\usepackage{amsmath}
\usepackage{amsthm}
\usepackage{amsfonts}
\usepackage{amssymb}
\usepackage{graphicx}
\usepackage{geometry}
\usepackage{xcolor}

\usepackage{caption}
\usepackage{subcaption}
\usepackage{float}

\usepackage{appendix}

\newcommand{\Uad}{U_\text{ad}}

\newcommand{\R}{\mathbb R}
\newcommand{\norm}[1]{\left\lVert#1\right\rVert}

\newtheorem{assumption}{Assumption}
\newtheorem{definition}{Definition}
\newtheorem{theorem}{Theorem}
\newtheorem{remark}{Remark}
\newtheorem{proposition}{Proposition}

\newtheorem{lemma}{Lemma}

\usepackage[ruled,vlined]{algorithm2e}
\usepackage{algorithmic}

\graphicspath{ {./images/} }

\title{Gradient-based parameter calibration of an anisotropic interaction model for pedestrian dynamics}
\author{Zhomart Turarov$^1$, Claudia Totzeck$^2$}
\date{%
	$^1$TU Kaiserslautern, e-mail: \texttt{turarov@mathematik.uni-kl.de}\\%
	$^2$BU Wuppertal, e-mail: \texttt{totzeck@uni-wuppertal.de}\\[2ex]%
	\today
}
\begin{document}
     
\maketitle

\begin{abstract}
We propose an extension of the anisotropic interaction model which allows for collision avoidance in pairwise interactions by a rotation of forces \cite{Claudia} by including the agents' body size. The influence of the body size on the self-organization of the agents in channel and crossing scenarios as well as the fundamental diagram is studied. Since the model is stated as a coupled system of ordinary differential equations, we are able to give a rigorous well-posedness analysis. Then we state a parameter calibration problem that involves data from real experiments. We prove the existence of a minimizer and derive the corresponding first-order optimality conditions. With the help of these conditions we propose a gradient descent algorithm based on  mini-batches of the data set. We employ the proposed algorithm to fit the parameter of the collision avoidance and the strength parameters of the interaction forces to given real data from experiments. The results underpin the feasibility of the method. 
\end{abstract}

\section{Introduction}
Mathematical modelling of pedestrian dynamics is of practical benefit in civil engineering \cite{Khelfa2021,Schadschneider2018,Schadschneider_etal_2009} for example in the design of complex architectures, e.g., stadiums, city centers, shopping malls, or for the management of large public events like festivals, concerts, pilgrimages, or manifestations \cite{Chraibi2018}. 
Capturing both, the individual and collective behaviours in pedestrian dynamics, is rather complex \cite{Bellomo2022,Castellano2009}. Many different approaches have been proposed in the literature: for example models based on magnetic forces proposed by S. Okazaki and S. Matsushita in which pedestrians are modeled as magnetic charges in a magnetic field \cite{Teknomo_2016}; the gas-kinetic model which treats pedestrians as molecules in liquefied gas \cite{Hoogendoorn_2001}; % In this model the pedestrian dynamics and velocity are unknown, instead, the statistical distribution of particles is known based on the Boltzmann equation ;  Another quick way to simulate pedestrian traffic is 
cellular automata \cite{BurgerWolfram,Burstedde_Klauck_2001,Dijkstra_2001}; models incorporating anticipative, rational behaviour \cite{Bailo2018,Christiani2015,Degond2013} and (smooth) sidestepping \cite{Festa2018,Claudia}. Another group of pedestrian models has emerged from the pioneering work on social forces \cite{Helbing_Molnar_1995} and can be classified as force-based \cite{Charibi2011,Claudia} and the overview given in \cite{Chen2018}. 

Most of these models share the property of being able to reproduce collective features such as lane formation in counterflow scenarios and travelling waves in crossing flows. Moreover, they can be used to study evacuation scenarios. On the other hand they strongly differ in their description. Indeed, some models, for the example the class of forces-based models have a sound mathematical description and allow for a statement in terms of a closed system of ordinary or partial differential equations. Others have a rather algorithmic structure because they require the solution of optimization problems to estimate for example the time-to-collision in very iteration. For the latter, a rigorous mathematical study of well-posedness is difficult.

Naturally, the modelling process is followed by an optimization or calibration procedure. For pedestrian dynamics the optimization of buildings, evacuation routes and traffic safety or the minimization of the occurence of high densities are of special interest \cite{Michael_J_S,Zhou2019}. Moreover, with the collection of data from real world experiments grew the interest in parameter fitting for the different pedestrian models \cite{Bode,Susana,Goettlich2021}.  

This work is in a similar spirit. First, we extend the anistropic model proposed in \cite{Claudia} by incorporating a body size for the interacting agents. This induces another dimension of volume exclusion in the model and aims at making the model more realistic. We study the influence of the body size on the formation of lanes and travelling waves as well as the fundamental diagram of the dynamics numerically and provide a rigorous study of the well-posedness of the interaction dynamics with and without body size employing standard ODE theory.
The second part of the article is concerned with the rigorous derivation of an gradient-based parameter calibration algorithm. We begin with the derivation of the first-order optimality system and propose a mini-batch gradient-descent algorithm for the calibration problem. 

In more detail, the article is organized as follows: the anisotropic model for pedestrian dynamics is extended by including the agents' body size in Section~\ref{sec:Microscopic_model}. Moreover, we study the influence of the body size on the collective behaviour and the fundamental diagram in there. Section~\ref{sec:par_calibration} is devoted to the parameter calibration problem. We begin with the statement of the problem, investigate its well-posedness and derive the corresponding first-order optimality conditions. Finally, the iterative gradient descent algorithm based on mini-batches for the parameter calibration is proposed in Section~\ref{sec:calibration_results}, where we show results obtained with this algorithm. We conclude with a summary of the results and an outlook on future projects.

\section{Microscopic model with body size}
\label{sec:Microscopic_model}
We include the body size in the  anisotropic model proposed in \cite{Claudia} as follows: let us consider a second order equation of motion with $N \in  \mathbb{N} $ agents. Their positions and velocities are denoted by  $x_{i}:[0,T] \rightarrow \mathbb{R}^2 $ and $v_{i}:[0,T] \rightarrow  \mathbb{R}^2, \  i=1,...,N.$ Moreover, the agents are assumed to have a body diameter $d>0$. 
This leads to the following interaction dynamics
\begin{subequations}
\label{eqn:state_equation}
\begin{align}
\label{eqn:state_equation_a}
\frac{d}{dt}x_{i}&=v_{i}, \\
\label{eqn:state_equation_b}
\frac{d}{dt}v_{i}&=   \tau\left( w_{i} - v_{i} \right) - \frac{1}{N} \sum_{j\neq i}M\left( v_{i}, v_{j}\right)K\left( d, x_{i}, x_{j}, v_{i}, v_{j}\right), \\
& {x_{i}(0) = x_0^i}, \quad  {v_{i}(0) = v_0^i}, \quad {i=1,...,N},
\end{align}
\end{subequations}
where  $K\left( d, x_{i}, x_{j}, v_{i}, v_{j}\right): \mathbb{R}^D \times \mathbb{R}^D \times \mathbb{R}^D \times \mathbb{R}^D \rightarrow \mathbb{R}^D$ is a pairwise interaction force between the agents $i$ and $j$. The rotation matrix $M\left( v_{i}, v_{j}\right)$ changes the direction of the interaction force. It reads
\begin{align}
\label{eqn:rotation_matrix}
M\left( v_{i}, v_{j}\right) = \begin{pmatrix} 
\cos\alpha_{ij} & -\sin\alpha_{ij} \\
\sin\alpha_{ij} & \cos\alpha_{ij} 
\end{pmatrix}, 
\quad 
\alpha_{ij} = 
\begin{cases}
\lambda \arccos  \frac{v_{i}\cdot v_{j}}{ \norm{v_{i} }\norm{v_{j} } }, & \text{if } v_i\neq 0, \ v_j\neq 0, \\
0,                                                                      & \text{else}
\end{cases}.  
\end{align}
In addition, the model includes a relaxation parameter $ \tau > 0$ which controls the adaption of the current velocity $v_i$ towards the given desired velocity $w_i \in \mathcal C([0,T],\R^2)$. The rotation of the force vectors induced by the matrix $M$ models a collision avoidance behaviour of the agents. The direction of the collision avoidance is controlled by the sign of the parameter $\lambda$. For $\lambda >0$ agents move to the right, to avoid a collison, for $\lambda <0$ the movement is directed to the left. See \cite{Claudia} for further details.

For notational convenience, the solution of the system is expresses by the vectors ${\mathbf x(t)=(x_1(t), ...,x_N(t))\in \mathbb{R}^{ND} }$ and 
$ \mathbf v(t)=(v_1(t), ...,v_N(t))\in \mathbb{R}^{ND} $ for $t\in[0,T].$

\begin{remark}
We can easily include obstacles or walls in the model, by describing them as artifical agents with fixed positions and fixed velocities and adding an additional interaction term similar to the interaction of the agents in \eqref{eqn:state_equation_b}.
\end{remark}

\subsection{Well-posedness}
In this section we study the well-posedness of the dynamics given in \eqref{eqn:state_equation}. We make the following assumptions on the interaction force $K\left( d, x_{i}, x_{j}, v_{i}, v_{j}\right)$ with $ i,j \in \{1,\dots,N\}.$ 

\begin{assumption}
	\label{as:well-posedness_1}
	The  interaction forces  $K\left(d,x_{i}, x_{j}, v_{i}, v_{j}\right)$  are locally Lipschitz and globally bounded w.r.t.~the positions $x_i, x_j$ and the velocities $v_i, v_j$.
\end{assumption}
\begin{assumption}
	\label{as:well-posedness_2}
	The gradients of interaction forces $\nabla K\left(d,x_{i}, x_{j}, v_{i}, v_{j}\right)$  exist, are locally Lipschitz and globally bounded w.r.t.~the positions $x_i, x_j$ and the velocities $v_i, v_j$.
\end{assumption}
\begin{remark}
Note that the first assumption is necessary to show to well-posedness of \eqref{eqn:state_equation} while we need the second assumption later on to obtain the well-posedness of the calibration problem.
\end{remark}

A key step to derive the well-posedness of the system is to establish the Lipschitz property of the right-hand side. In particular, the rotation of the force vector is of interest.
\begin{lemma}\label{lem:LipInteraction}
For the rotation of the force term with $v_i, v_j, v_k,v_\ell \in \R^2$ and $d\ge0$ it holds
\begin{align*}
| M\left( v_{i}, v_{j}\right) & K\left( d, x_{i}, x_{j}, v_{i}, v_{j}\right) - M\left( v_{k}, v_{\ell}\right)K\left( d, x_{k}, x_{\ell}, v_{k}, v_{\ell}\right) | \\
&\le L_1| (v_i,v_j) - (v_k,v_\ell) |  + L_2 | (x_i,x_j) - (x_k,x_\ell) | 
\end{align*}
for some Lipschitz constants $L_1, L_2>0.$
\end{lemma}
\begin{proof}
	 We introduce the short hand notation $$v^1 = \left( v_{i}, v_{j}\right), \  v^2 = \left( v_{k}, v_{\ell}\right), \quad x^1 = \left( x_{i}, x_{j}\right), \ {x^2 = \left( x_{k}, x_{\ell} \right)}. $$ In the following we omit the dependence on $d$. Hence we rewrite the left-hand side of the Lipschitz condition as 
	\begin{align}
		\label{eqn:ineq}
			 |M(v^1)K(x^1,& v^1) - M(v^2)K(x^2, v^2)| \\
			&\leq  |K(x^1, v^1)|_\infty |M(v^1) - M(v^2) | + |M(v^2)|_\infty\, |K(x^1, v^1) -  K(x^2,v^2) |. \nonumber
	\end{align}
	We estimate the first term on the right-hand side of \eqref{eqn:ineq} as
	\begin{equation}
	|M(v^1) - M(v^2)| \leq \left| \int_{0}^{1} \nabla M(v^2 + s(v^1 - v^2) ) ds \right| \left|v^1 - v^2\right|,
	\end{equation}
	where
	\begin{equation*}
		\nabla M(v^1 ) = \left(\frac{dM}{d {v}_i}(v^1 ), \frac{dM}{d {v}_j}(v^1)\right),
	\end{equation*}
	with
	\begin{gather*}
		\label{eqn:dif_rotation_matrix}
		\frac{dM}{d {v}_i}= \begin{pmatrix} 
			-\sin\alpha & -\cos\alpha \\
			\cos\alpha & -\sin\alpha 
		\end{pmatrix} \cdot \frac{d\alpha}{d {v}_i} , \\
		\ \frac{d\alpha}{d {v}_i} = 
		\begin{cases}
			-\lambda  \frac{1}{ \sqrt{(\norm{ {v_i} }\norm{{v_j} } ) ^2 - \langle {v_i}, {v_j} \rangle^2}}\left( {v_j} - \langle {v_i}, {v_j} \rangle \frac{{v_i}}{\norm{{v_i}}^2}\right) , & \text{if } {v_i}, {v_j}\neq 0 \\
			0,                                                                      & \text{else}
		\end{cases}  .
	\end{gather*}
	Analogously, we derive $\frac{dM}{d {v}_j}$. Each element of $\nabla M(v^2 + s(v^1 - v^2) ) $ is bounded, for a detailed proof of the boundedness see Appendix $\ref{app:AppendixB}$. Note that $K$ is globally bounded by Assumption \ref{as:well-posedness_1}. Moreover, $K$ is locally Lipschitz by Assumption \ref{as:well-posedness_1} which allows us to estimate the second term on the right-hand side of \eqref{eqn:ineq}. Altogether, this proves the existence of the Lipschitz constants $L_1 \text{ and } L_2 $ as desired.
\end{proof}

 \begin{theorem}[Existence and Uniqueness]
 	\label{lem:uniq_sol_st_eq}
 Let Assumption \ref{as:well-posedness_1} hold. Further we assume $w_i\in \mathcal C([0,T],\R^2), i=1,\dots,N$ and $A,R \ge 0, a,r >0$ and $\lambda \in [-1,1]$.
 
 Then system ($\textit{\ref{eqn:state_equation}}$) admits a unique solution $\mathbf x \in \mathcal C^1([0,T], \R^{2N}), \mathbf v\in \mathcal C^1([0,T], \R^{2N}).$
 \end{theorem}
 
 \begin{proof}
 	On the basis of Assumption~\ref{lem:LipInteraction}  and  Lemma~\ref{lem:LipInteraction}, the result can be directly obtained with the help of Picard-Lindelöf theorem.
 
 \end{proof}

\subsubsection{Body size} \label{Numerical_schemes_state_system}
In the previous discussion the body size is an abstract parameter. To give more details, we consider a variation of the Morse potential \cite{Orsogna_Chuang}  leading to the interaction potential $$U(d, |x_i - x_j|) = R\cdot e^{\frac{d-\norm{x_i - x_j}}{r}} - A\cdot e^{\frac{d-\norm{x_i - x_j}}{a}}$$ leading to the forces $K$ given by
\begin{equation}
	\label{eqn:interaction_force}
	K(d, |x_i - x_j|) = \left(\frac{A}{a} \cdot e^{\frac{d-\norm{x_i - x_j}}{a}} - \frac{R}{r}\cdot e^{\frac{d-\norm{x_i - x_j}}{r}} \right) \cdot \frac{x_i - x_j}{\norm{x_i - x_j}}.
\end{equation}

\begin{remark}
Note that the forces depending on the body size as given in \eqref{eqn:interaction_force} satisfy the assumptions of Theorem~\ref{lem:uniq_sol_st_eq}. Hence, we have existence and uniqueness of solutions for the case with body size $d>0$ as well.
\end{remark}

\subsection{Numerical studies}
For the numerical studies of the model we draw random initial positions with uniform distribution in the domain and set initial velocities with respect of the desired direction of motion. We set the initial velocity to the desired velocity. Then we solve ($\ref{eqn:state_equation}$) with a variant of the leap frog scheme \cite{Claudia}. Indeed, the relaxation terms are solved implicitly and the interaction is solved explicitly as given by
\begin{equation}
	\label{eqn:leap_frog_scheme}
	\begin{split}
		&x_i^{k^{'}} = x_i^{k} + \frac{\Delta t}{2}v_i^{k}, \quad \qquad \qquad \qquad \qquad \qquad \qquad \quad \ \ \  v_i^{k^{'}} = (v_i^{k} + \Delta t \cdot w_i)/ (1 + \Delta t) \\
		&v_i^{k+1} = v_i^{k^{'}} + \Delta t \cdot  \frac{1}{N} \sum_{j\neq i}M(v_i^{k^{'}}, v_j^{k^{'}}) \cdot K(x_i^{k^{'}}, x_j^{k^{'}}), \quad x_i^{k+1} = x_i^{k^{'}} + \frac{\Delta t}{2}v_i^{k+1}, \ i=1,\dots,N,
	\end{split}
\end{equation}
where, $\Delta t$ denotes the step size of the time discretization. 

\subsubsection{Influence of the body size}
In the following we provide some numerical results showing the influence of the body size on  lane  formation in the corridor and crossing scenario, respectively. 
The first experiment simulates the movement of two oncoming streams of pedestrians along a spacious corridor.  The group of blue agents moves from left to right with desired velocity $w_\text{blue}=(0.7,0)^T$, wereas the red group of agents moves from right to left with desired velocity $w_\text{red}=(-0.7,0)^T$.  We consider $N_\text{blue}$ blue and $N_\text{red}$ red agents. Hence, the total number of pedestrians in the corridor is $N = N_\text{blue} + N_\text{red}$. The initial positions of the pedestrians $x(0) = \textbf{x}_0$ and their initial velocities $v(0) = \textbf{v}_0$ are illustrated in Figure \ref{fig:init_positions_corridor}.

To assure that the pedestrians do not leave the scenario, we add reflective and periodic boundary conditions. In the corridor case the black lines (top and bottom) in Figure \ref{fig:init_positions_corridor} show reflective boundaries.
We model the avoidance of wall contact, by reflecting the velocity vector of an agent that would step outside of the domain in the next time step. The behavior of reflection from the wall is the same as in \cite{Claudia}. The light blue lines illustrate periodic boundaries. Blue agents leaving the domain at the boundary on the right, enter again from the left. Analogously for the red agents. 

\begin{figure}[H]
	\begin{subfigure}{0.6\textwidth}
		\includegraphics[scale=0.67]{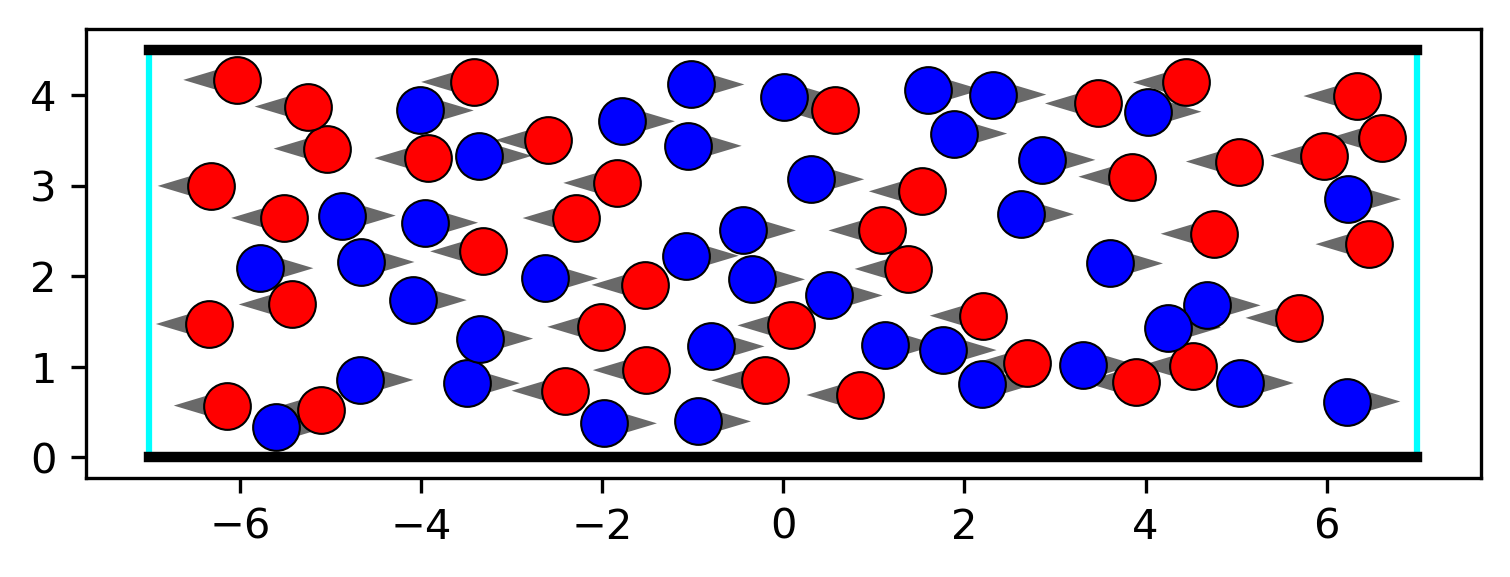} 
		\caption{Corridor}
		\label{fig:init_positions_corridor}
	\end{subfigure}
	\begin{subfigure}{0.4\textwidth}
		\includegraphics[scale=0.6]{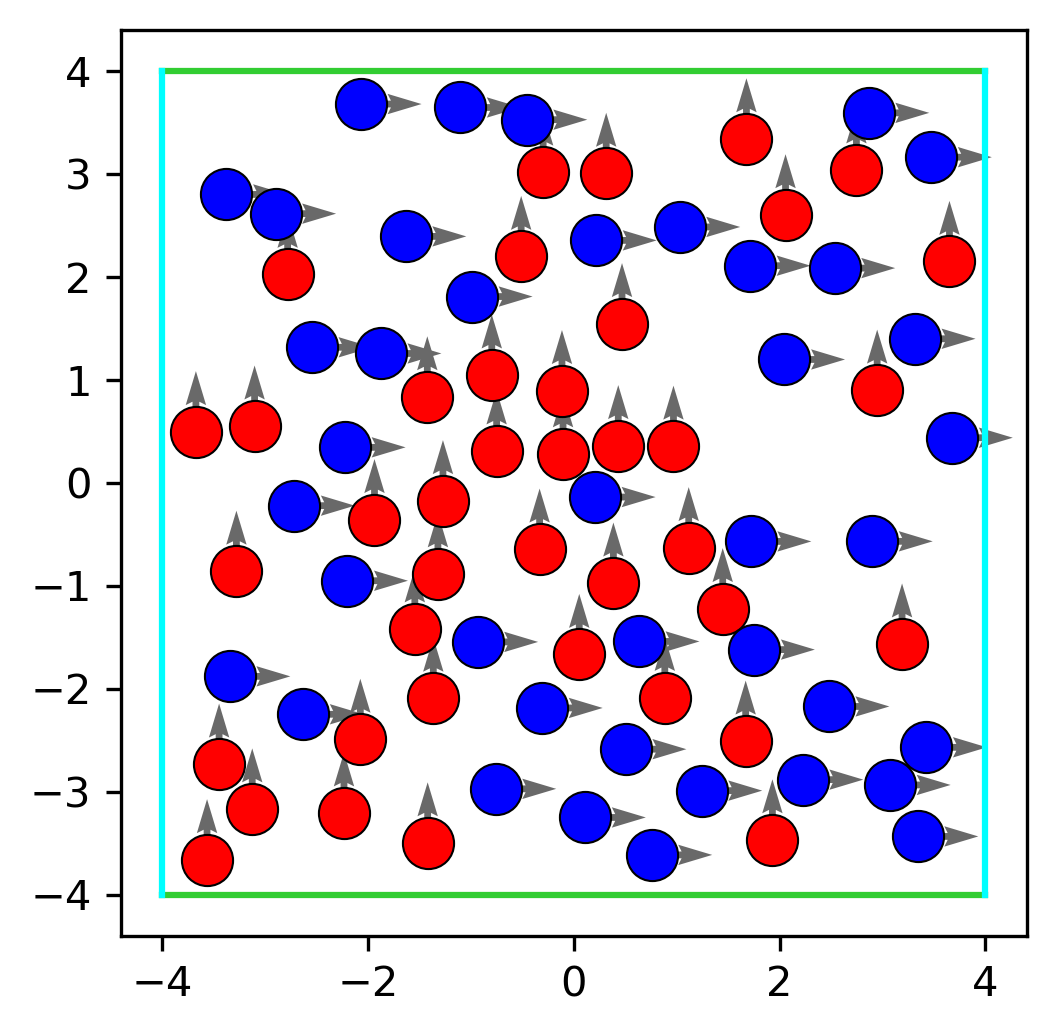}
		\caption{Crossing}
		\label{fig:init_positions_crossing}
	\end{subfigure}
	\caption{Initial positions and initial velocity vectors for the two scenarios with parameters $N=80,\  N_\text{blue} = 40, \ N_\text{red} = 40$,  $d = 0.2.$}
	\label{fig:initial_positions}
\end{figure}

In the second scenario we consider two groups of pedestrians at a crossing. Here, the blue group of pedestrians moves from left to right with desired velocity $w_\text{blue}=(0.7,0)^T$, and red group of pedestrians moves from bottom to top with desired velocity $w_\text{red}=(0,0.7)^T$. In total, there are 
$N = N_\text{blue} + N_\text{red}$ pedestrians. The initial positions and initial velocity vectors are presented in Figure \ref{fig:init_positions_crossing}. There, light blue and green lines represent periodic boundaries for blue and red agents respectively. The bottom and top green lines are inflow and outflow for red agents, left and right blue lines are inflow and  outflow for blue agents. Altogether, this results in Algorithm~\ref{alg:leap_frog} for the state system \eqref{eqn:state_equation}.

\medskip

\begin{algorithm}[H] \label{alg:leap_frog}
	\SetAlgoLined
	\KwIn{initial positions $\textbf{x}_0$ and velocities $\textbf{v}_0$ of pedestrians, model parameters}
	\While{$ t < T$}{
		update the positions and velocities of the agents using \eqref{eqn:leap_frog_scheme} \\
		check reflective boundary condition: if an agent is about to leave through a reflective border, its velocity is changed in the direction pointing inwards\\ 
		check periodic boundary condition: if an agent leaves through a periodic boundary, its position is adjusted due to the periodic boundary conditions\\
		$t = t + \Delta t$
	}
	\caption{ state system}
	\KwResult{trajectories and velocities at the time steps
	}
\end{algorithm}

\subsubsection{Study for different body sizes}

To analyse the simulation results for different body sizes, we fix values for the force parameters, and desired velocities of each pedestrian. The parameters are chosen to satisfy the stability ranges of the interaction force discussed in \cite{Orsogna_Chuang}. In fact, in the range $R/A > 1$ and $r/a < 1$ the interaction force $K$ is repulsive in a short-range  and attractive in a long-range. This allows the distance between pedestrians to be maintained.  Even if we include body size into the interaction force, it remains attractive in a short-range and repulsive in a long-range.  

\begin{figure}[ht!]
	\begin{subfigure}{0.5\textwidth}
		\includegraphics[scale=0.375]{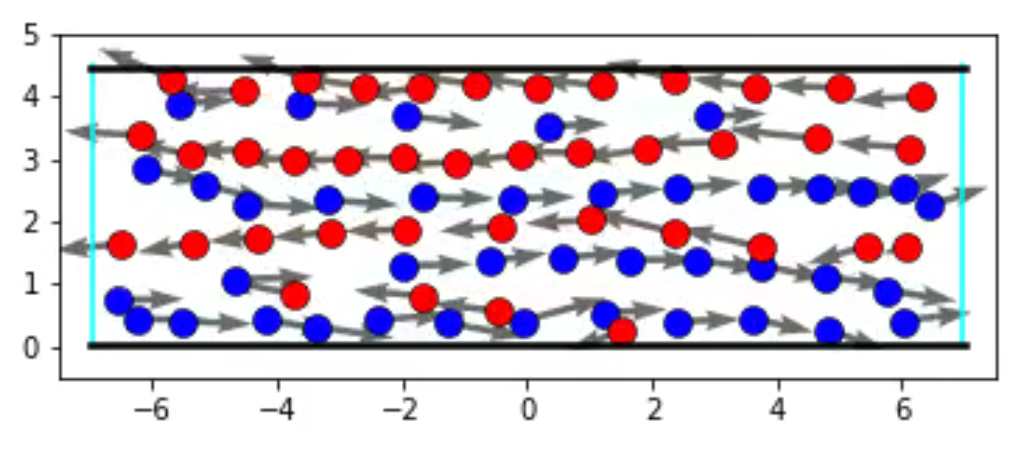} 
		\caption{Body diameter $d=0.4$}
		\label{fig:corr_br02}
	\end{subfigure}
	\begin{subfigure}{0.5\textwidth}
		\includegraphics[scale=0.51]{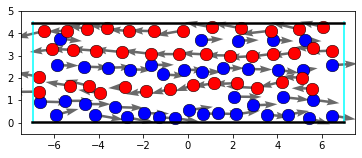} 
		\caption{Body diameter $d=0.46$}
		\label{fig:corr_br022}
	\end{subfigure}
	\begin{subfigure}{0.5\textwidth}
		\includegraphics[scale=0.37]{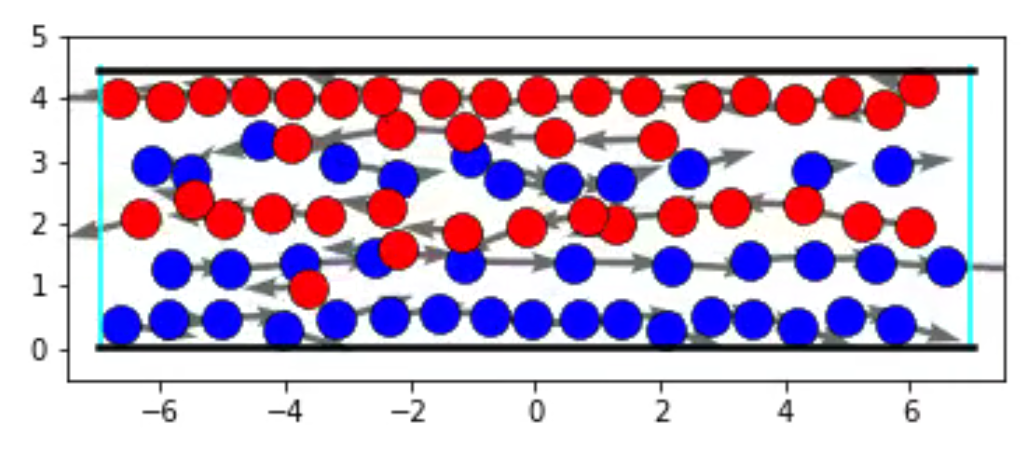}
		\caption{Body diameter $d=0.5$}
		\label{fig:corr_br_025}
	\end{subfigure}
	\begin{subfigure}{0.5\textwidth}
		\includegraphics[scale=0.5]{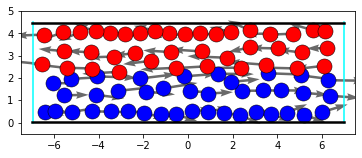}
		\caption{Body diameter $d=0.6$}
		\label{fig:corr_br_03}
	\end{subfigure}
	\caption{Simulation results in the corridor by different body size of pedestrians at time $T = 35$. On each simulation we fix parameters: $A=5, R=20, a=2, r=0.5, \lambda = 0.25$. Desired velocities for red and blue agents are  $w_\text{red}=(-0.7, 0)^T$ and $w_\text{blue}=(0.7, 0)^T$  respectively. Time step in the Leap-Frog Scheme is $\Delta t = 0.00625$.}
	\label{fig:lane_formation_by_body_size}
\end{figure}

Figure~\ref{fig:lane_formation_by_body_size} shows simulation results of the corridor scenario for different body sizes. The results indicate a relation between the body size and the number of lanes formed. The smaller the body size, the more lanes are obtained. The parameters used for the simulation are reported in the caption of the figure. Similar results are found for the crossing scenario in Figure~\ref{fig:lane_formation_cross_by_body_size}. Again, the smaller the body size, the more lanes are formed.

\begin{figure}[ht!]
	\begin{subfigure}{0.32\textwidth}
		\includegraphics[scale=0.5]{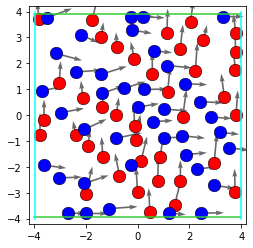} 
		\caption{Body diameter $d=0.4$}
		\label{fig:cros_br_02}
	\end{subfigure}
	\begin{subfigure}{0.32\textwidth}
		\includegraphics[scale=0.5]{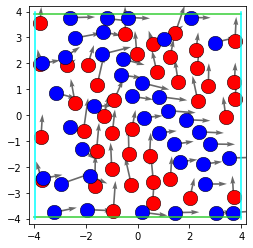} 
		\caption{Body diameter $d=0.46$}
		\label{fig:cros_br_023}
	\end{subfigure}
	\begin{subfigure}{0.32\textwidth}
		\includegraphics[scale=0.5]{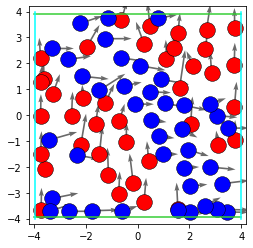}
		\caption{Body diameter $d=0.5$}
		\label{fig:cros_br_025}
	\end{subfigure}
	\caption{Simulation results at the crossing by different body size of pedestrians at time $T = 35$. On each simulation we fix parameters: $A=5, R=20, a=2, r=0.5, \lambda = 0.25$. Desired velocities for red and blue agents are  $\vec{w}_\text{red}=(0, 0.7)^T$ and $\vec{w}_\text{blue}=(0.7, 0)^T$  respectively.  Time step in the Leap-Frog Scheme is $\Delta t = 0.00625$.}
	\label{fig:lane_formation_cross_by_body_size}
\end{figure}

In all simulations we see the formation of so-called traffic lanes. This formation seems to be independent of the choice of the  random initial positions and velocities. It is interesting to note that even though every pedestrian is guided by simple rules for movement and interaction, phenomena arise that go beyond the behaviour of single pedestrians. Such phenomena of the self-organization are manifested in many multi-agent systems \cite{Dirk_Helbing}. They were reported in many articles concerning the movement of pedestrian flows \cite{Sieben_etal_2017,Zhang_etal_2019}, which speaks in favour of the proposed model. Moreover, we want to emphasize that not only the body size can influence the number of lanes. In fact, the choice of the width of the corridor, number of agents and attraction and repulsion force parameters can change the formation of lanes as well. For the force parameters this is shown exemplarily in Figure~\ref{fig:lane_formation_corr_by_diff_pars}. 

\begin{figure}[ht!]
	\begin{subfigure}{0.5\textwidth}
		\includegraphics[scale=0.4]{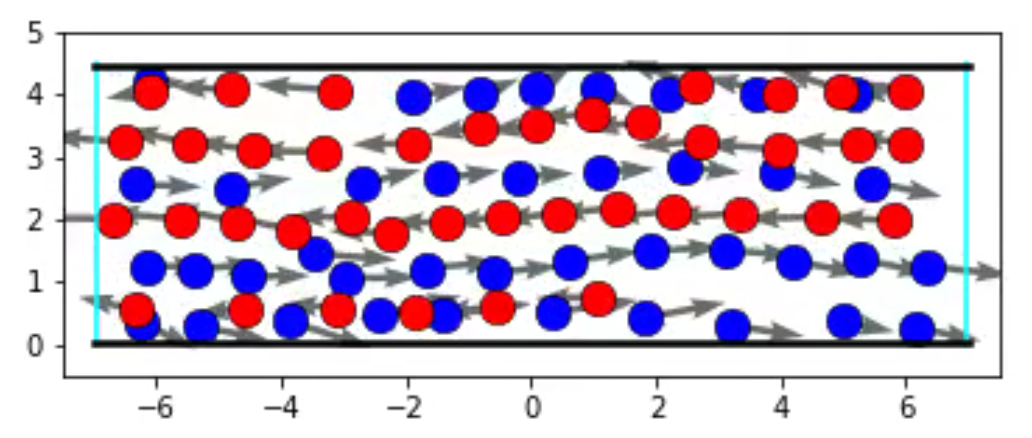} 
		\caption{ $A=3, R=15, a=2, r=0.5, \lambda = 0.25$}
		\label{fig:corr_br025_A3R15a2r05}
	\end{subfigure}
	\begin{subfigure}{0.5\textwidth}
		\includegraphics[scale=0.4]{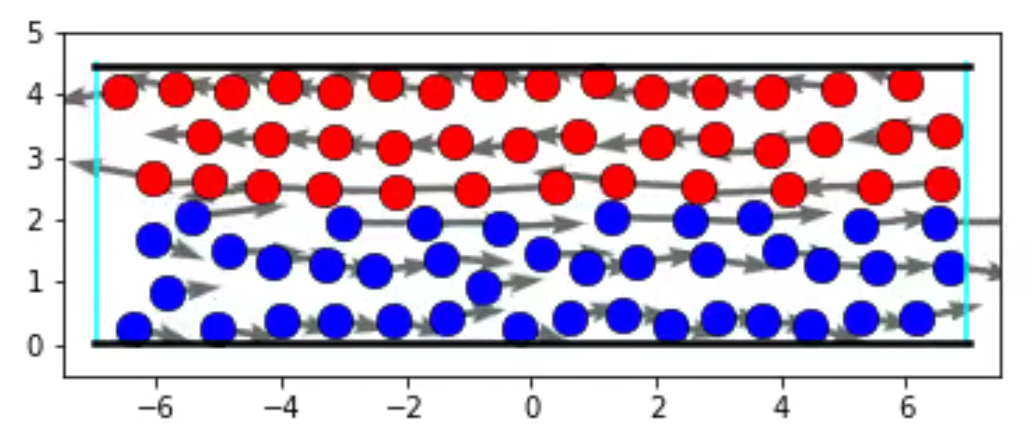} 
		\caption{ $A=5, R=25, a=1.2, r=0.3, \lambda = 0.25$}
		\label{fig:corr_br025_A5R25a12r03}
	\end{subfigure}
	
	\begin{subfigure}{0.5\textwidth}
		\includegraphics[scale=0.4]{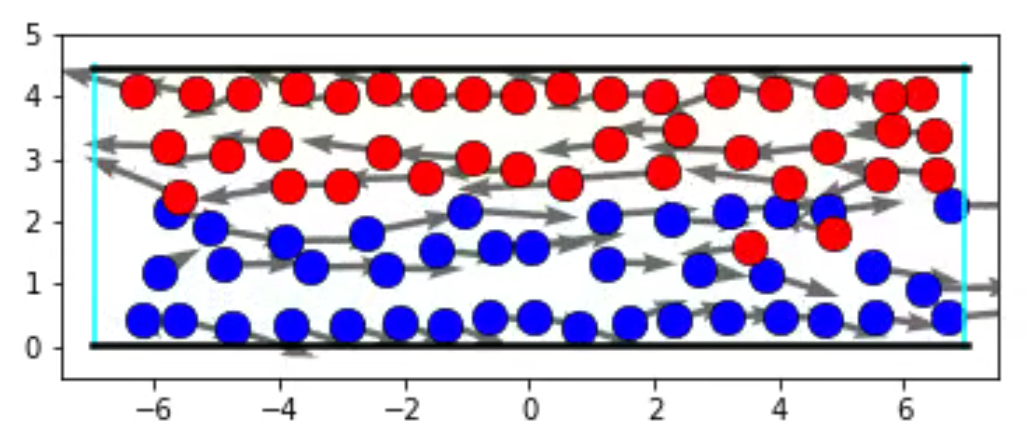} 
		\caption{ $A=4, R=26, a=2, r=0.5, \lambda = 0.25$}
		\label{fig:corr_br025_A4R26a2r05}
	\end{subfigure}
	\begin{subfigure}{0.5\textwidth}
		\includegraphics[scale=0.4]{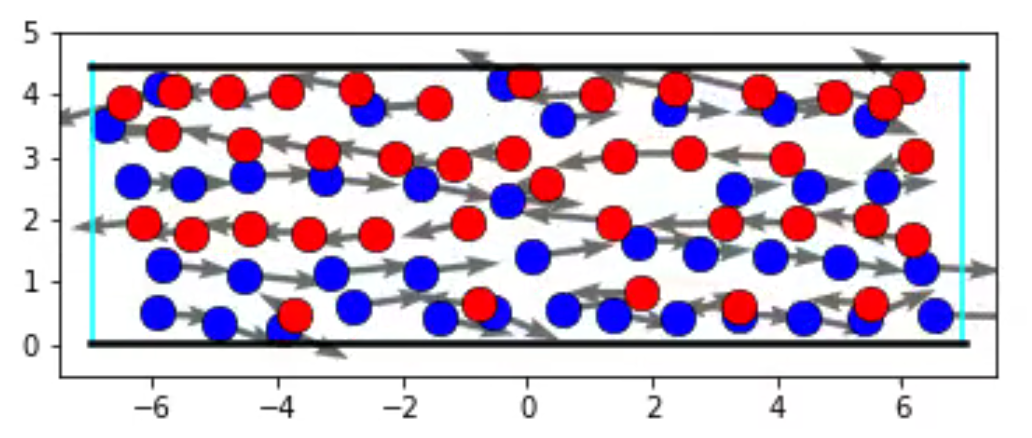} 
		\caption{ $A=5, R=25, a=2.4, r=0.6, \lambda = 0.25$}
		\label{fig:corr_br025_A5R20a24r06}
	\end{subfigure}
	
	\caption{Simulation results in the corridor by different force parameters at time $T = 35$. On each simulation body diameter of the agents are fixed: $d=0.5$. Time step in the Leap-Frog Scheme is $\Delta t = 0.00625$.}
	\label{fig:lane_formation_corr_by_diff_pars}
\end{figure}	

As the body size, ratio of repulsion and attraction amplitudes, and size of the corridor have similar effects in terms of volume exclusion, we suspect from these studies that the volume exclusion is the main driver of the lane formation process.

\subsubsection{Fundamental diagram}
Often fundamental diagrams are employed to analyse crowd motion models \cite{Schadschneider_etal_2009, Zhang_etal_2019}. Main objective is the relationship of speed and density \cite{ Seyfried_etal_2005, Steffen_Seyfried_2010, Zhang_etal_2019} which we study for the 
\begin{figure}[hbt!]
	\begin{subfigure}{0.48\textwidth}
		\includegraphics[	width=\textwidth,
		height =0.4\textwidth]{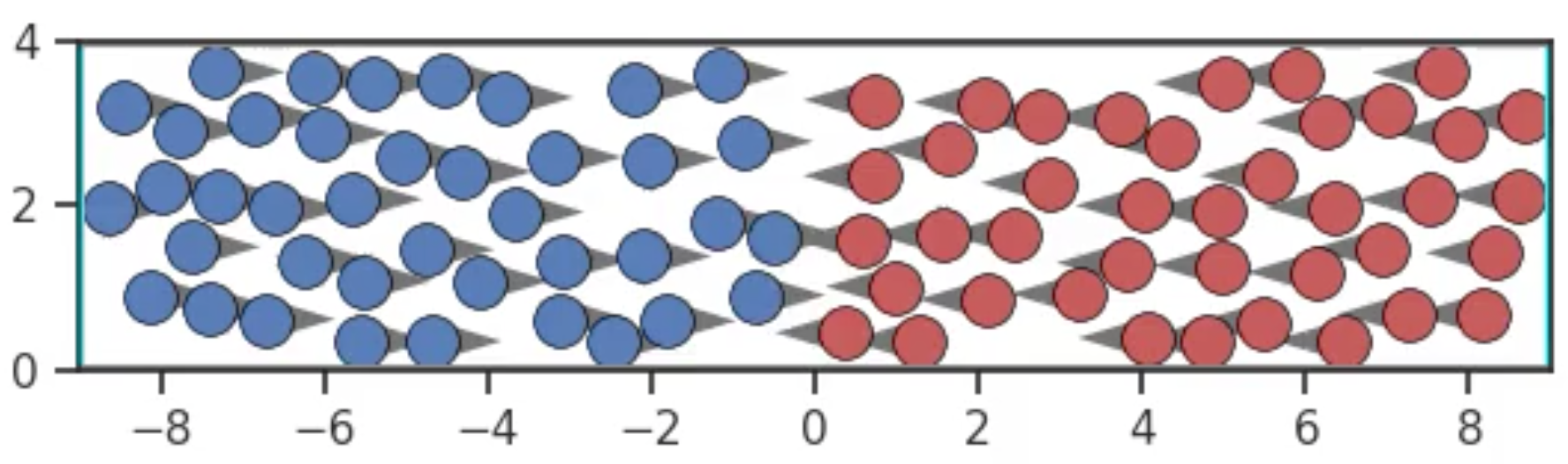}
	\end{subfigure}
	\begin{subfigure}{0.48\textwidth}
		\includegraphics[	width=\textwidth,
		height =0.4\textwidth]{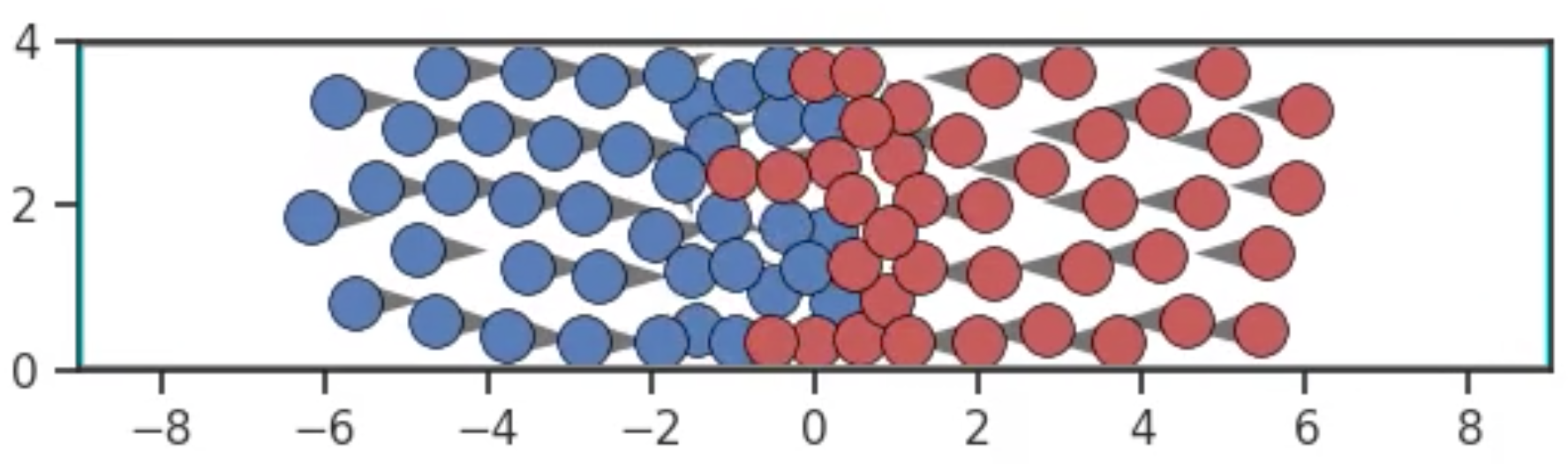}
	\end{subfigure}
	\caption{Bi-directional pedestrian flow at time $T=0$ and $T=5s$.}
	\label{fig:state_pos}
\end{figure} 
bi-directional flow simulated with the model described in Section \ref{sec:Microscopic_model} on the corridor with $17m$ length and $4m$ width, see Figure \ref{fig:state_pos} for an illustration. The density approximation is realized with the help of Voronoi diagrams as proposed in \cite{Cao_Seyfried_2017, Steffen_Seyfried_2010}.  Initially, agents move with their desired walking speed until they slow down (or speed up) due to interaction forces. Most interactions take place in the centre of the domain, it is therefore the focus of our interest. 

\begin{figure}[hbt!]
	\begin{subfigure}{0.48\textwidth}
		\includegraphics[width = \linewidth, 
		height =0.4\textwidth]{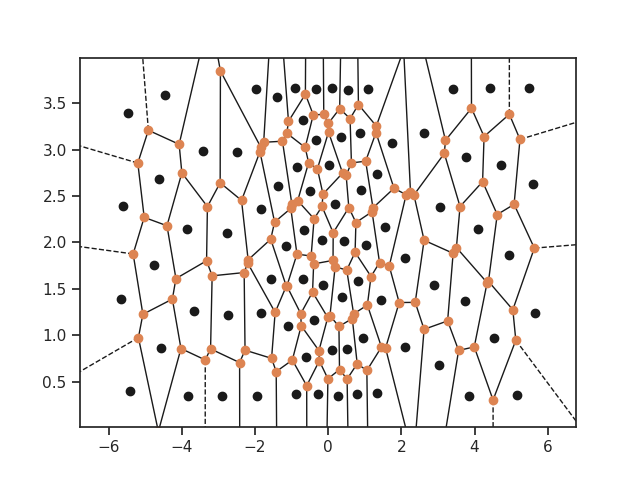} 
		\caption{\tiny Voronoi diagram}
		\label{fig:voronoi3_ts_5}
	\end{subfigure}
	\begin{subfigure}{0.48\textwidth}
		\includegraphics[width = \linewidth, 
		height =0.4\textwidth]{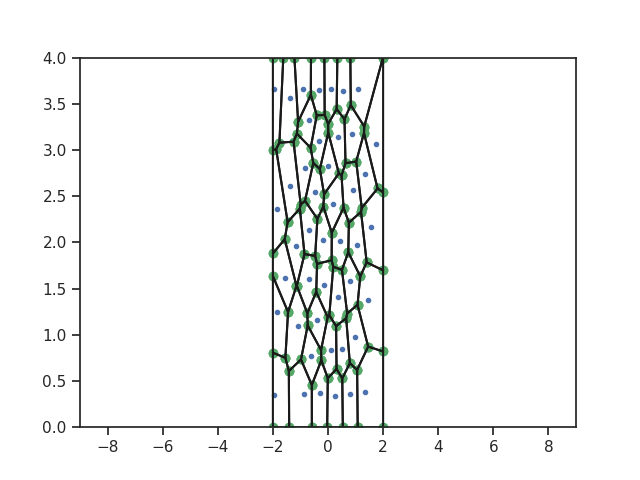}
		\caption{\tiny Bounded Voronoi diagram $\Omega = [-2, 2]\times[0, 4]$}
		\label{fig:bounded_voronoi3_ts_5}
	\end{subfigure}
	\caption{Voronoi diagrams of bi-directional pedestrian flow at $T=5s$.}
	\label{fig:voronoi}
\end{figure}

To approximate the density we construct Voronoi cells based on the positions of agents. The Voronoi diagram allows separating the area into computational grids (polygons) based on the triangulation of the computational area, in particular, the Delaunay triangulation \cite{Shewchuk_2002}. Every point on the Voronoi diagram has a region that is closer to it than any other \cite{Steffen_Seyfried_2010}. In Figure \ref{fig:voronoi3_ts_5} we see Voronoi cells of all agents. Figure \ref{fig:bounded_voronoi3_ts_5} shows the same cells, but only the ones in the region of interest $\Omega = [-2, 2]\times[0, 4].$ Outside of the domain $\Omega$ agents are involved in less interactions and walk approximately with their desired velocities. The model parameters in these simulation are set $\lambda=0.07$, $A = 0.2$, $R = 2$, $a=1$, $r=0.1$. 

\begin{figure}[H]
	\begin{subfigure}{0.48\textwidth}
		\includegraphics[	width=\textwidth,
		height =0.75\textwidth]{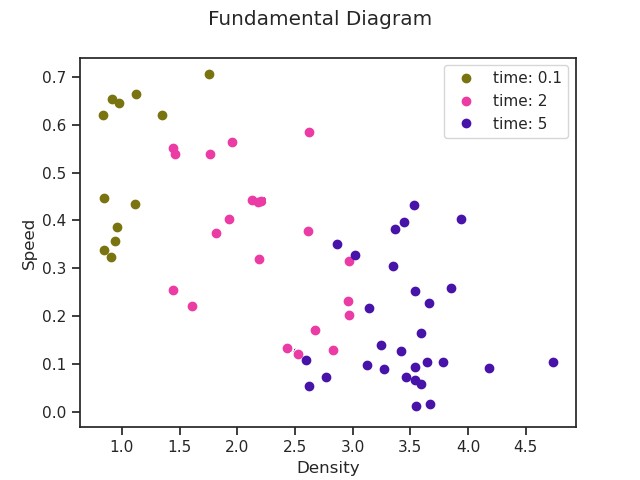}
		\caption{ $d=0.5$}
	\end{subfigure}
	\begin{subfigure}{0.48\textwidth}
		\includegraphics[	width=\textwidth,
		height =0.75\textwidth]{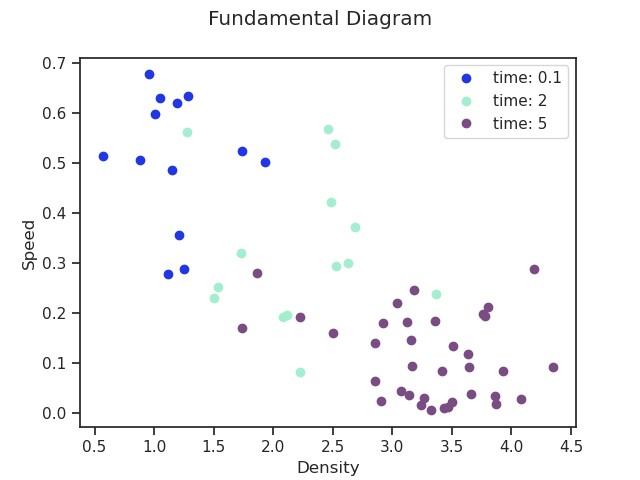}
		\caption{ $d=0.6$}
	\end{subfigure}
	\caption{Fundamental Diagrams of bi-directional pedestrian flow.  }
	\label{fig:density_vel}
\end{figure}

We employ the method proposed in \cite{Cao_Seyfried_2017, Steffen_Seyfried_2010} to compute the density as follows 
$$
\rho_{xy} (t)= \left\{\begin{split}
	1/A_{i}(t) \quad if (x,y) \in A_{i}(t), \\
	0 \quad \quad else \quad\quad\quad\quad\quad\\ 
\end{split}\right.
$$
where $A_i(t)$ gives the Voronoi cell area for agent $i$, and $\rho_{xy}$ is the density distribution of the space. Voronoi cells are computed using the Python, Scipy module, with the Voronoi and ConvexHull methods. Note that the velocities are given by the integration of \eqref{eqn:state_equation}.

In Figure \ref{fig:density_vel} we illustrate the relationship of density and speed at different time points. At the beginning, the agents move with the desired speed in regions with lower density leading to weak interaction forces. When the two groups meet, we see that as the density increases and at the same time the speed  decreases as expected. The body size seems to have minor influence on this effect.

\medskip

This ends our analysis and numerical study of the model. In the following sections we are concerned with its parameter calibration based on real data \cite{data_archive}. We begin with the statement of the calibration problem and analyse its well-posedness.

\section{Analysis of the parameter calibration problem}
\label{sec:par_calibration}
In this section, we state the parameter calibation problem and analyse the existence of minimizers. Then we lay the theoretical ground for the formulation of a gradient-descent method by deriving the first-order optimality conditions, the existence of adjoint states and finally the identification of the gradient of the reduced cost functional. This prepares the formulation of a gradient descent-based calibration algorithm that will be employed in the numerical section to fit the parameters to real data from the BaSiGo experiment carried out in Düsseldorf in 2013 \cite{Cao_Seyfried_2017, data_archive} in the next section.

Let us begin with the statement of the parameter calibration problem. Table~\ref{tab:model_parameters} shows all model parameters. For the calibration we focus on the scaling factor $\lambda$ involved in the collision avoidance process and the force strengths $A$ and $R.$ For fixed $1 \gg\epsilon > 0$, we define the set of admissible parameters 
\[
\Uad = \Big\{(\lambda, A, R) \in [-1+\epsilon, 1-\epsilon] \times [0,A_{\max}] \times [0,R_{\max}] \Big\} 
\] 
and we want to find ${u:= (\lambda, A, R)} \in \Uad$ such that the model trajectories fit the real trajectories best.
We thus consider the cost functional
\begin{equation*}
	J(y, u) := \int_{0}^{T} \frac{\sigma_1}{2N} \sum_{i=1}^{N}  \norm{x_{i}(t) - x_{i}^{\text{data}}(t)}^2_2 dt + \frac{\sigma_2}{2} \norm{u - u_\text{ref}}^2_2 \label{eqn:cost_functional}%\nonumber 
\end{equation*}
with $x^\text{data}$ given trajectory data from experiments. The first term measures the distance of the trajectories resulting from the model to the real trajectories from the data. The second term of the cost functional penalizes the distance of the parameters to some given reference parameters $u_\text{ref}.$ In case there are no reference values available, we set $\sigma_2 = 0.$

\begin{table}[H]	
	\begin{center}
		\caption{Model parameters}
		\label{tab:model_parameters}
		\begin{tabular}{|c|l|c|l|}   
			\hline
			$\lambda$ & scaling factor (collision angle) & 
			$\tau$ & relaxation parameter (desired velocity) \\
			$A$ & attractive force amplitude &
			$R$ & repulsive force amplitude \\
			$a$ & attractive force range &
			$r$ & repulsive force range \\
			$w$ & desired velocity  &
			$d$ & body size of pedestrian \\
			\hline
		\end{tabular}
	
	\end{center}
\end{table}

To study the well-posedness we use the following notion of optimality:
\begin{definition}
We call $u \in U_{ad}$ \emph{optimal}, if it is a solution to the optimization problem
\begin{equation}
\label{opt_problem} \tag{P}
\min_{(y,u) \in Y\times U_{ad}} J(y, u)   \quad \quad \text{subject to \eqref{eqn:state_equation}}. 
\end{equation}
\end{definition}
\noindent
Note that the calibration problem is constrained by the ODE system without boundary conditions. The boundary conditions are only incorporated in the numerical simulations to reflext the domain of the experiments appropriately.

\medskip

 In the following we consider the spaces $Y$ and $U$ given by
\[
	Y = [H^1([0,T] , \mathbb{R}^{DN}) \times H^1([0,T] , \mathbb{R}^{DN}) ], \quad U = [-1,1] \times \mathbb{R} \times \mathbb{R}.
\]
Note that both, $Y$ and $U$ are Hilbert spaces and $\Uad \subset U$ is closed. For notational convenience we define the state vector $y=(x, v) \in Y$ and the state operator 
\begin{equation}
	\label{eqn:State_oper}
	e: Y \times U \rightarrow Z^*, \qquad
	e(y,u) = \begin{pmatrix} 
		\frac{d}{dt} y-  F(y,u) \\
		y(0) - \textbf{y}_0 
	\end{pmatrix},
\end{equation}
where $F(y,u)$ is the vector containing the right-hand side of ($\ref{eqn:state_equation}a$) and ($\ref{eqn:state_equation}b$), respectively.

\subsection{Well-posedness of the parameter calibration problem}

The proof of the existence of an optimal parameter set for the calibration problem will be based on the following to Lemmata which are concerned with the boundedness of the states with respect to the control parameters and the weak continuity of the state operator $e$.

\begin{lemma}[Boundedness]  \label{lem:boundedness}
Let $w_i \in \mathcal C([0,T],\R^2)$ for all $i=1,\dots,N$ and Assumption~\ref{as:well-posedness_1} hold and suppose the interaction forces are given by \eqref{eqn:interaction_force}.
For given $u\in \Uad$ there exists $C>0$ depending only on the body size $d$ and $$ \bar w := \max\limits_{i} \sup\limits_{t\in[0,T]} |w_i(t)|,$$ such that the solution $y\in Y$ with $e(y,u) = 0$ satisfies
\[
\norm{y}_Y \le C(1+ |u|) .
\]  
\end{lemma}	
\begin{proof}
We begin with the estimate for $\norm{y}_{L^2(0,T,\R^2)}.$ It holds
\begin{equation*}
	|x(t)|^2 \le  2|x_0|^2+ 2\int_0^t |v(s)|^2 ds, \qquad
	|v(t)|^2 \le 2|v_0|^2 + 2\int_0^t |w(s) - v(s)|^2 + Ce^d |u|^2 ds.
\end{equation*}	
Hence, using Gronwall Lemma we obtain $|y(t)|^2 \le |u|^2 e^{Ct}$ and integration over $[0,T]$ yields $\norm{y}_{L^2(0,T,\R^2)} \le C_1 |u|.$ For the time derivatives, we find
\[
|\frac{d}{dt} x(t)|^2 \le  |v(s)|^2, \qquad
|\frac{d}{dt} v(t)|^2 \le 2 |w(s) - v(s)|^2 + 2Ce^d |u|^2.
\]
Integration over $[0,T]$ leads to $\norm{\frac{d}{dt} y}_{L^2(0,T,\R^2)} \le  C(1 + |u|).$ The two estimates together give the result.
\end{proof}

\begin{lemma}\label{lem:weakCont}
The state operator 
\[e: Y \times U \rightarrow Z^*, \qquad e(y,u) = \begin{pmatrix} 
	\frac{d}{dt} y-  F(y,u) \\
	y(0) - \textbf{y}_0 
\end{pmatrix} 
\]
is weakly continuous.
\end{lemma}
\begin{proof}
We need to show that $e(y_k,u_k) \rightharpoonup e(\hat{y},\hat{u})$ as $k\rightarrow \infty$, which can be reformulated as follows.

For any given test function $\varphi \in C_c^1(Z)$, we need to obtain the following convergence property
$$
\lim\limits_{k\rightarrow \infty} \int \langle e(y_k,u_k), \varphi \rangle dt \rightarrow \int \langle e(\hat{y},\hat{u}), \varphi \rangle dt.
$$

We estimate
\begin{align*}
&\lim\limits_{k\rightarrow \infty} \int_{0}^{t} \left[ \frac{d}{dt}y_k - \frac{d}{dt}\hat{y} - F(y_k,u_k) + F(\hat{y},\hat{u}) \right] \varphi dt  \\
&=: \lim\limits_{k\rightarrow \infty} \int_{0}^{t} \left[ \frac{d}{dt}y_k - \frac{d}{dt}\hat{y} + y_k - \hat y + M_{\lambda_k}(v) K_{A_k,R_k}(d,x) -  M_{\hat\lambda}(v) K_{\hat A,\hat R}(d,x)\right] \varphi dt \\
&=  \lim\limits_{k\rightarrow \infty} \int_{0}^{t} \left[ \frac{d}{dt}y_k - \frac{d}{dt}\hat{y} + y_k - \hat y \right] \varphi dt  \\
&\quad+ \lim\limits_{k\rightarrow \infty} \int_{0}^{t} \left[ \big(M_{\lambda_k}(v) - M_{\hat\lambda}(v)\big) K_{A_k,R_k}(d,x) + M_{\hat\lambda}(v) \big( K_{A_k,R_k}(d,x)- K_{\hat A,\hat R}(d,x) \big)\right] \varphi dt 
\end{align*}
Clearly, the first integral tends to zero for $k\to\infty$ by the weak convergence of $y \in Y.$ The first term in the second integral tends to zero by continuity of the map $\lambda \mapsto M_\lambda(v)$ and the second term of the second integral tends to zero by the continuity of the interaction force with respect to $A$ and $R$. Altogether, this yields the desired result.

\end{proof}

\begin{theorem}
There exists at least one solution $(y^*, u^*) \in Y \times  U_{ad}$ to \eqref{opt_problem}.
\end{theorem}

\begin{proof}
The cost functional $J$ is bounded from below and the state system is well-posed, so there exists
$$
m=\inf\limits_{(y,u) \in  Y \times U_{ad} } J(y,u).
$$

Let $(u_{k}) \in U_{ad}$ be a minimizing sequence. The sequence $(u_k) \subset \Uad$ is bounded and by reflexivity of $U$ it has a weakly convergent subsequence (not relabeled) with limit $\hat{u}$.
 By Lemma~\ref{lem:boundedness} we obtain the boundedness of $(y_k)$ and, again by reflexivity, the  existence of $\hat y$ such that
\begin{equation}
\label{eqn:weak_limit}
u_k \rightharpoonup \hat{u}  \text{ in } U_{ad} \quad \text{and} \quad y_k \rightharpoonup \hat{y}  \text{ in }  Y \text{ as }k\rightarrow \infty . 
\end{equation}
The weak continuity of $e(y,u)$ shown in Lemma~\ref{lem:weakCont} implies
$$
\norm{e(\hat{y}, \hat{u})  } \leq \liminf\limits_{k\rightarrow \infty} \norm{e(y_k,u_k) } = 0.
$$
Hence $\hat y$ is the solution to the state equation with parameters $\hat u.$
By the weak lower semicontinuity of the norm, we obtain
$$
J(\hat{y}, \hat{u}) \leq \liminf\limits_{k\rightarrow \infty} J(y_k,u_k) = m,
$$
which allows us to conclude that $(\hat{y}, \hat{u})$ is a minimizer of \eqref{opt_problem}.
\end{proof}

\begin{remark}
Because of the nonlinearity of the state equation we cannot expect uniqueness of the optimal control.
\end{remark}

Having the existence of an optimal solution, we proceed with the derivation of the first-order necessary conditions, which will form the basis of the identification algorithm.

\subsection{First-order necessary conditions}
\label{FONC}

We introduce the dual pairing
\begin{flalign}
\left\langle e(y, u), (\xi, \eta)\right\rangle_{Z, Z^*}   
=\int_{0}^{T} \left(\frac{d}{dt}  y -F(y,u) \right) \cdot \xi(t) dt+ (y(0)-\textbf{y}_0) \cdot \eta,
\end{flalign}
where $\xi, \eta$ are the Lagrange multipliers in the Banach space $$Z = [L^2([0,T] , \mathbb{R}^{DN}) \times L^2([0,T] , \mathbb{R}^{DN}) \times \mathbb R^{ND} \times \mathbb R^{ND} ]$$ that represent the space of the adjoint states. Here, $\xi = (\xi_x, \xi_v),$ $\eta = (\eta_x, \eta_v)$ and $(\xi, \eta) \in Z$. %We denote the dual of $Z$ as $Z^*$. 
The space of states $Y$ and the set of controls $U$ are defined at the beginning of in Section~\ref{sec:par_calibration}.

We formally derive the first-order optimality system with the help of the Lagrangian corresponding to the constrained parameter calibration problem given by
\begin{equation}
\label{eqn:Lagrangian}
\mathcal{L}: Y\times U \times Z \rightarrow \mathbb{R}, \qquad 
\mathcal{L}(y, u, \xi, \eta) = J(y, u) + \left\langle e(y, u), (\xi, \eta)\right\rangle_{Z, Z^*}. 
\end{equation}

\subsubsection{Derivation of the adjoint system and the optimality condition}
\label{sec:Adjoint_MP}
Solving $d \mathcal{L}(y, u, \xi, \eta) = 0$ yields the first-order optimality condition \cite{Rene_Pinnau}. We begin with the computation of the directional derivatives of \eqref{eqn:Lagrangian} with respect to the state $y.$ For notational convenience we consider $x$ and $v$ separately. 

First, we obtain Gâteaux derivatives of the objective function
\begin{equation*}
d_{x_i} J(y, u)[h_{x_i}] = \frac{\sigma_1}{N}  \int_{0}^{T}  (x_{i}(t) -x_{i}^{data}(t) )h_{x,i} (t) dt,\qquad
d_{v_i} J(y, u)[h_{v_i}] = 0.
\end{equation*}

Now we derive the directional derivatives w.r.t.~the positions of agents for the second part of Lagrange functional
\begin{flalign*}
%\label{eqn:E_total_x}
&d_{x_i}\left\langle e(y,u), (\xi, \eta)\right\rangle[h_{x_i}]  =d_{x_i}\left\langle e_i(y,u), (\xi, \eta)_i\right\rangle[h_{x_i}] + \sum_{\substack{j=1\\j\neq i}}^{N}d_{x_i}\left\langle e_j(y,u), (\xi, \eta)_j\right\rangle[h_{x_i}].&
\end{flalign*}

 Using integration by parts, we obtain 
\begingroup
\allowdisplaybreaks
\begin{flalign}
\label{eqn:E_total_x1}
d_{x_i} \left\langle e(y,u), (\xi, \eta)\right\rangle & [h_{x_i}]  = \int_{0}^{T}  h_{x_i}^{'}\xi_{1,i}(t) + \frac{1}{N} h_{x_i} \sum_{\substack{j=1\\j\neq i}}^{N}\biggl(M(v_i,v_j)d_{x_i}K(x_i,x_j) \biggr)^T\xi_{2,i}(t) dt    \\
&\qquad\qquad+ \int_{0}^{T} h_{x_i}\frac{1}{N} \sum_{\substack{j=1\\j\neq i}}^{N} \left[M(v_j,v_i) d_{x_i}K(x_j,x_i) \right]^T\xi_{2,j}(t)dt +  h_{x_i}(0)\eta_{1,i}.& \nonumber
\end{flalign}
Since $M(v_i,v_j) = M(v_j,v_i)$ is symmetric and $d_{x_i}K(x_i,x_j) = -d_{x_i}K(x_j,x_i)$ by the radial symmetry of $K$, we obtain
\begin{flalign*}
d_{x_i} \langle &e(y, u), (\xi, \eta)\rangle[h_{x_i}]  = h_{x,i}(0)\eta_{1,i} \\
\qquad  &+\int_{0}^{T} \biggl[ h_{x,i}^{'}(t)\xi_{1,i}(t) + h_{x,i}(t)  \frac{1}{N}  \sum_{\substack{j=1\\j\neq i}}^{N}\biggl(M(v_i,v_j)d_{x_i}K(x_i,x_j) \biggr)^T (\xi_{2,i}(t) - \xi_{2,j}(t))    \biggr] dt .
\end{flalign*}
\endgroup

Similarly, we obtain the derivative in direction $[h_{v_i}].$ Indeed, by using integration by parts we find
\begingroup
\begin{equation}
\label{eqn:E_total_v}
d_{v_i}\left\langle e(y, u), (\xi, \eta)\right\rangle  [h_{v_i}]  =d_{v_i}\left\langle e_i(y, u), (\xi, \eta)_i\right\rangle[h_{v_i}] + \sum_{\substack{j=1\\j\neq i}}^{N}d_{v_i}\left\langle e_j(y, u), (\xi, \eta)_j\right\rangle[h_{v_i}].
\end{equation}
To simplify the notation we introduce the operator $d_{v_i}M^*(v_i,v_j)$ resulting from matrix reformulations, see Appendix~\ref{app:AppendixA} for more details. We get
\begin{align*}
d_{v_i}\langle &e(y, u), (\xi, \eta)\rangle  [h_{v_i}]  = h_{v_i}(0)\eta_{2,i} + \int_{0}^{T}  h_{v_i}^{'}\xi_{2,i}(t) -  h_{v_i}\xi_{1,i}(t) 
+ \tau h_{v_i}\xi_{2,i}(t)\; dt \\
&+  \int_{0}^{T} h_{v_i} \biggl[ \frac{1}{N} \sum_{\substack{j=1\\j\neq i}}^{N}d_{v_i} M^*(v_i,v_j) K(x_i,x_j) \xi_{2,i}(t) 
+ \frac{1}{N}\sum_{\substack{j=1\\j\neq i}}^{N} d_{v_i}M^*(v_j,v_i) K(x_j,x_i)\xi_{2,j}(t) \biggr] dt.
\end{align*}

Moreover, the directional derivatives w.r.t.~the control are given by
\begin{flalign}
	&d_u J(y, u)[h_u] = \sigma_2 (u -u_{ref})h_{u}, &
\end{flalign}
\begin{flalign}
	&d_{u}\left\langle e(y, u), (\xi, \eta)\right\rangle[h_{u}]   = -\int_{0}^{T} h_{u} \sum_{i} \left[d_{u} F(y_i,u)\right]^T   \xi_{2,i}(t)dt.&
\end{flalign}

\begin{remark}
	For the specific choice \eqref{eqn:interaction_force}, as we choose the interaction forces for the numerical results, the derivatives read
	\begin{flalign*}
		&\frac{dF(y_i,u) }{d\lambda} = \begin{cases}
			-\frac{1}{N}  \sum \limits_{j \neq i} \begin{pmatrix} 
				-\sin\alpha_{ij} & -\cos\alpha_{ij} \\
				\cos\alpha_{ij}  & -\sin\alpha_{ij} 
			\end{pmatrix} \arccos \frac{v_i \cdot v_j}{\norm{v_i} \norm{v_j}} \cdot K(x_i, x_j), & \text{for } v_i, v_j \neq 0, \\ \\
			\vec{0}_{2\times 1},                                                                      & \text{else}
		\end{cases},&\\
		&\frac{dF(y_i,u) }{dA} = -\frac{1}{a\cdot N} \sum_{j\neq i} M(v_i,v_j) \cdot e^{\frac{d-\norm{x_i-x_j}}{a}} \cdot \frac{x_i - x_j}{\norm{x_i-x_j}},&\\
		&\frac{dF(y_i,u) }{dR} = \frac{1}{r\cdot N} \sum_{j\neq i} M(v_i,v_j) \cdot e^{\frac{d-\norm{x_i-x_j}}{r}} \cdot \frac{x_i - x_j}{\norm{x_i-x_j}}.&
	\end{flalign*}
\end{remark}

\endgroup

\subsubsection{Existence of  adjoint states}
\label{ExistenceAdjoints}

The proof of the existence of adjoint states is based on Corollary 1.3 in \cite{Rene_Pinnau}, which we give in Appendix~\ref{app:existenceAdjoint} for completeness.

\begin{theorem}
	Let $w_i \in \mathcal C([0,T], \R^2), i=1,\dots,N$ be given, the Assumption~\ref{as:well-posedness_1} and \ref{as:well-posedness_2} hold and $u^* \in \Uad$ be an optimal solution of Problem (\ref{opt_problem}) and let $y^* \in Y$ such that $e(y^*,u^*)=0$. Then there exist an adjoint state  $p^*=(\xi^{*}, \eta^*) \in Z^*$ such that the following optimality conditions hold
	\begin{align*}
	&\langle e(y^*,u^*), p \rangle_{Z,Z^*} = 0 &&\forall p \in Z^*, \\
	&\langle L_y(y^*, u^*,p^*) , h \rangle_{Y^*,Y} = 0 &&\forall h \in Y^*, \\
	&u^* \in \Uad, \langle L_u(y^*,  u^*, p^*) , u-u^* \rangle_{U^*,U} \ge 0, &&\forall u\in \Uad.
\end{align*}
\end{theorem}
\begin{proof}
We check requirements given in Assumpion~\ref{ass:cor}:
\begin{itemize}
	\item [(A1)] $ \Uad =  [-1+\epsilon, 1-\epsilon] \times [0,A_{\max}] \times [0,R_{\max}]  \subset U$ is nonempty, convex
and closed.
	
	\item [(A2)] We first note that $U,Y,Z$ are Banach spaces. Further, $J$ is of tracking type and therefore Fr\'echet differentiable \cite{Fredi_Troltzsch}. We are left to show Fr\'echet differentiability of the state operator $e(y,u): Y \times U \rightarrow Z.$

	In Section~\ref{sec:Adjoint_MP} we computed the first variations of $e(y,u)$ with w.r.t.~$y$ and $u$ by
	$$
	d_{y} \left \langle e(y, u)[(h_y)], (\xi, \eta)\right\rangle = \lim\limits_{\epsilon \rightarrow 0} \frac{1}{\epsilon}\left \langle e(y+\epsilon h_y, u) - e(y,u), (\xi, \eta)\right\rangle,
	$$
	and obtained \eqref{eqn:E_total_x1} and \eqref{eqn:E_total_v}. There, continuity of linear terms follows directly from the definition. We show the continuity for nonlinear terms: 
	
	Let, the sequence $y_k:=(x_k, v_k) \subset Y$ have the limit $y_k$ such that $x_k \rightarrow x \text{ and } v_k \rightarrow v $ as $k \rightarrow \infty$. Using Assumption~\ref{as:well-posedness_2} we show continuity of nonlinear terms in \eqref{eqn:E_total_x1}. Indeed, for the $i$-th component, it holds
	\begin{flalign*}
	&\int_{0}^{T}    \frac{1}{N} h_{x_i} \sum_{\substack{j=1\\j\neq i}}^{N}\biggl(M(v_{k_i},v_{k_j})d_{x_i}K(x_{k_i},x_{k_j}) - M(v_i,v_j)d_{x_i}K(x_i,x_j) \biggr)^T\ \left(\xi_{2,i}- \xi_{2,j}\right)dt \\
	&\leq L_1\cdot L_2 \int_{0}^{T} \frac{1}{N}\sum_{\substack{j=1\\j\neq i}}^{N}\biggl( \norm{v_{k_i} - v_i} + \norm{v_{k_j} - v_j} + \norm{x_{k_i} - x_i} + \norm{x_{k_j} - x_j} \biggr)\norm{ \xi_{2,i} - \xi_{2,j}}dt \\
	&\leq L_1 \cdot L_2 \int_{0}^{T} \norm{y_{k_i} - y_i} \norm{ \xi_{2,i} - \xi_{2,j}}dt,
	\end{flalign*}
	where $L_1, L_2$ are Lipschitz constants. Analogously, we show the continuity for the nonlinear terms of (\ref{eqn:E_total_v})  using Assumption \ref{as:well-posedness_1} and Appendix \ref{app:AppendixB}. These, yields to the continuity of the state operator $e(y,u)$ w.r.t.~$y$. The continuity w.r.t.~$u$, can be concluded from the continuity of the $\lambda \mapsto M,$ and the linearity of the force term w.r.t~$A$ and $R.$ Altogether, this proves the continuity of $e$ as desired.
	
	\item [(A3)] By Theorem \ref{lem:uniq_sol_st_eq}  the state system $e(y,u) = 0$ has an unique solution $y=y(u) \in Y$ for all $u \in V \subset U$ a neighbourhood of $\Uad.$

	\item [(A4)] We have to show that $e_y(y(u), u) \in \mathcal{L}(Y,Z)$ has a bounded inverse for all $u \in V \supset \Uad$.
We can write $d_{y}e(y, u)[h]$ in general form
$$d_{y}e(y, u)[h] = \frac{d}{dt}h(t) +  d_yF(y, u)h(t),$$
where $d_yF(y, u)$ is integrable in $t$ over every finite interval $I \subset [0,T]$ thanks to Assumption~\ref{as:well-posedness_2}.

We consider $d_{y}e(y, u)[h] = r$ for arbitrary $r\in Z^*.$ By Carathéodory's existence theorem, we get for every $r \in Z^*$ a unique solution \cite{Fornasier_2014}, namely $h=d_y e(y,u)^{-1}[r]$. Using $d_{y}e(y(u), u)$, we obtain
$$
\norm{h(t)} \leq \norm{h(0)} + \int_{0}^{T} \norm{r(s)}ds +C\exp{\left(\int_{0}^{T}\norm{h(s)}ds\right)}, \quad t\in[0,T]
$$
and with the help of Gronwall's Lemma we get the boundedness of the inverse of $d_{y}e(y(u), u)$.
\end{itemize}
Altogether, the requirements of Assumption~\ref{ass:cor} are satisfied, and thus the Proposition~\ref{prop:adjoint} (see Appendix~\ref{app:existenceAdjoint}) yields the existence of an adjoint state.
\end{proof}

\begin{remark}
Note that assuming more regularity of the adjoint states, for instance $Z=Y$, we may establish the strong formulation of the adjoint system given by
\begin{flalign}
	\label{eqn:Adjoint_system}
	\xi^{'}_{1,i}(t) = & \frac{\sigma_1}{N} (x_{i}(t) - x_{i}^{data}(t)) +  \frac{1}{N} \sum_{\substack{j=1\\j\neq i}}^{N}\biggl(M(v_i,v_j)d_{x_i}K(x_i,x_j)\biggr)^T(\xi_{2,i}(t) - \xi_{2,j}(t)) ,  \nonumber \\
	\xi^{'}_{2,i}(t) = &-\xi_{1,i}(t) + \tau \xi_{2,i}(t)   - \frac{1}{N} \sum_{\substack{j=1\\j\neq i}}^{N} d_{v_i} M^*(v_i,v_j)K(x_i,x_j) \xi_{2,i}(t) \nonumber \\
	&- \frac{1}{N} \sum_{\substack{j=1\\j\neq i}}^{N} d_{v_i} M^*(v_j(t),v_i(t))K(x_j(t),x_i(t))  \xi_{2,j}(t)&
\end{flalign}
supplemented with the terminal conditions $ \xi_{1,i}(T) =0 , \  \xi_{2,i}(T) =0.$ The strong form will be employed for the numerical results in Section~\ref{sec:calibration_results}.
\end{remark}

\subsubsection{Gradient of the reduced cost functional}
\label{GradientReducedCostFunctional}

To minimize the objective function we aim to apply gradient descent algorithm. In order to determine the  gradient we define the control-to-state operator $ \mathcal{F}: U \rightarrow Y $, and  introduce the reduced cost functional as ${\hat{J}(u):= J(\mathcal{F}(u), u)}$.  Using $e(y,u) = 0$, we obtain
$$
0=d_{y}e(\mathcal{F}(u), u)[d\mathcal{F}(u)] + d_{u} e(\mathcal{F}(u),u)). 
$$

Taking the derivative of the Lagrangian with respect to the state $y$, we get
$$
d_{y}e(y,u)^{*}\xi = -d_{y}J(y,u).
$$

With these, we can compute the G\^{a}teaux derivative of the reduced cost functional in the direction $h_{u} \in U$, and obtain
\begin{align*}
d\hat{J}(u)[h_{u}] &= \langle d_{y}J(y,u), d\mathcal{F}(u)[h_{u}] \rangle + \langle d_{u}J(y,u),h_{u}\rangle \\
&= \langle d_{u}e(y,u)^{*}\xi, h_{u} \rangle +  \langle d_{u}J(y,u), h_{u}\rangle = d_{u}\mathcal{L}(y,u,\xi)[h_u].&
\end{align*} 
Note that we have already computed $ d_{u}J(y,u)$ in Section~\ref{sec:Adjoint_MP}.
This allows us to identify the gradient of the reduced cost functional as
\begin{flalign}
\label{eqn:gradient_reduced_CF}
\nabla \hat{J}(u) = \gamma (u -u_{ref})
-\int_{0}^{T}  \sum_{i} \left[d_{u} F(y_i,u)\right]^T   \xi_{2,i}(t)dt. 
\end{flalign}

\section{Calibration algorithm and results}\label{sec:calibration_results}
To calibrate the control parameters we use real data from the BaSiGo experiment carried out in Düsseldorf in 2013 \cite{Cao_Seyfried_2017, data_archive}. In the following we denote the trajectories from experimental data by $x_{i}^\text{data}:[0,T] \rightarrow  \mathbb{R}^D, {i=1\dots N}.$  For the corridor case, the data we take from file "bi\_corr\_400\_a\_02.txt" in \cite{data_archive}. These show the positions of the pedestrians in the domain $\Omega  = [-6, 6]\times[0, 4.2], $ over time  $t \in [0, 150]$ seconds. For the crossing case, the data we take from file "CROSSING\_90\_E\_2.txt" in \cite{data_archive}. This file provides the positions of the pedestrians in the crossing corridors $\Omega  = {([-5, 5]\times[-1.5, 2])\cap ([-1.2, 2]\times[-5, 5])}, $ over time $t \in [0, 283]$ seconds. However, for the calibration procedure, we use only 8-second intervals for both scenarios. 

\subsection{Numerical schemes and steepest descent algorithm}

In general, the nonlinearities make it rather difficult to solve the optimality system all at one. We therefore opt for an iterative approach to compute the gradient of the calibration problem. Indeed, we first solve the state system \eqref{eqn:state_equation} as considered in Section \ref{Numerical_schemes_state_system}.  Then, we integrate the adjoint system \eqref{eqn:Adjoint_system} with the help of a second-order Runge-Kutta method backward in time. Here, we use the same time steps as for the state problem and transform the time via $s=T-t$ to recover an initial value problem. With the state solution and the adjoint solution, we calculate the gradient using \eqref{eqn:gradient_reduced_CF}, where the integral is approximated with the trapezoidal rule.

We apply a steepest descent algorithm to update the control parameters in every iteration
\begin{flalign}
	\label{eqn:update_control}
	u^{k+1} = u^{k} - \beta_k \cdot \nabla  \hat{J}(u) 
\end{flalign}
where $u^{k}$ denotes the control on current time step, and $\nabla  \hat{J}(u) $ denotes the descent direction and $\beta_k \in \R^3$ is a positive scaling vector. The complete optimization procedure is summarized in Algorithm~\ref{alg:gradient_descent}.

Since we cannot assume that the data and the model are a perfect match, we employ a stochastic gradient descent approach using mini-batches \cite{Li_Zhang_etal_2014}. To obtain the mini-batches from the data trajectories with are given on the interval $[0,T]$ we split the interval into $M$ mini-batches, each of size $b_i$,  $b_i \subset [t_1^i, t_2^i], i=1 \dots m$. Then we randomly select $m<M$ mini-batches for each of the gradient steps.

In more detail, at each iteration we compute the gradients for all the time intervals $b_i, \ { i=1 \dots m}$, and approximate the gradient using the average
$$
{ \nabla  \hat{J}(u)  = \frac{1}{m} \sum_{i=1}^{m} \nabla  \hat{J}_{b_i}(u) }
$$
to update the control via \eqref{eqn:update_control}. The stopping criterion of the calibration algorithm is based on the relative error between the previous and current cost function value denoted by $\epsilon_\text{rel}.$

\begin{algorithm}[H] \label{alg:gradient_descent}
	\SetAlgoLined
	\KwIn{Initial data $(\textbf{x}_0,\textbf{v}_0)$ of the pedestrians and all parameter values, initial guess for the parameters $u_0$}
	\While{$\epsilon_\text{rel} > 10^{-2}$}{
		Solve state system ($\ref{eqn:state_equation}$)\;
		Solve adjoint system ($\ref{eqn:Adjoint_system}$)\;
		Calculate gradient of the reduced cost functional ($\ref{eqn:gradient_reduced_CF}$)\;
		Update control parameters ($\ref{eqn:update_control}$)\;
	}
	\caption{Steepest Descent Algorithm}
	\KwResult{calibrated controls $\bar u$, trajectories $(\bar x, \bar v)$ and their cost $J(\bar y,\bar u)$.}
\end{algorithm}

\subsection{Numerical results}

In this section we discuss numerical results generated with Algorithm $\ref{alg:gradient_descent}$ using experimental data from the Pedestrian Dynamics Data Archive $\cite{data_archive}$. In particular, we retrieve the trajectories from video recordings showing bi-directional and cross-directional flows.

We fix the pedestrian body size $d=0.5$, the velocity scaling $\tau = 1$, attractive potential range $a=1$, repulsive potential range $r=0.3$,  number of pedestrians $N=84$, and desired velocity of agents that go from right to left and from left to right are respectively $w_\text{red} = (-0.7, 0)^T$ and $w_\text{blue}=(0.7, 0)^T$. The value $0.7$ is the average velocity which we extracted from the experimental data. The desired velocities for the crossing scenario are again estimated from the experimental data and set to $w_\text{red} = (0, 1.2)^T$ and $w_\text{blue}=(1.2, 0)^T$.

The time step in the Leap Frog scheme and second order Runge-Kutta method is set to $\Delta t = 0.00625$. Simulations are done in the time interval $t\in[0,8]$. The gradient calculated with $m = 50$ mini-batches of length $|b_{i}| =\Delta t $.  

The initial positions of the agents $\mathbf x_0$ coincide with the initial positions of the experimental data, which are distributed in the domain $[-6,6]\times[0,4.2]$. As the initial velocity of the pedestrians, we set the average velocity from the experimental data. We probably induce some error here, as we do not have the exact values from the data. 

As initial guess of the control parameters we take $u_0=(0, 0, 40)$ for both experiments. We set the regularisation parameters in the cost functional to $\sigma_1 = 1$ and $\sigma_2 = 0,$ we therefore do not need to choose reference values for the control. In this setting we perform 100 iterations with Algorithm $\ref{alg:gradient_descent}$. The step size parameter is set to 
 ${\beta_k = (20, 4000, 4000)}$ to account for the different ranges of the control values.

Figure~\ref{fig:cost_corridor} illustrates the decrease of the cost functional for the corridor case. It starts from approximately 6.51 in the first iteration and terminates around 5.11 in the last iteration. In Figure \ref{fig:cost_crossing} we see the evolution of the cost for the crossing scenario from 9.03 to 7.34. We note that the cost values of the corridor are smaller than the ones of the crossing case. This indicates that the proposed model captures effects of counter-flow better than effects of crossing flow. In addition, we detect that the evolution of the cost is smoother for the crossing case. This could  be for the same reason.

The parameters of the  calibration procedure reached to optimal values $\lambda = -0.17$, $A = 5.56$, and $R=29.23$ for the simulation in the corridor. For simulation at the crossing optimal values reached $\lambda = -0.18$, $A = 10.41$, and $R=29.02$. Interestingly, the optimal $\lambda$ is negative and in the same range in both scenarios. This indicates that the pedestrians in the data set have a tendency to move to the left to avoid collisions. Moreover, the repulsion strengths are similar for both scenarios, but the values for the attraction strengths differ. The difference of the attraction strengths may arise from the post-interaction behaviour of the model. Simulations as reported in \cite{Claudia} show that two interacting agents move on with slightly shifted positions after an collision avoiding interaction. We suspect that this has more impact on the results in the crossing than the corridor scenario. So far, however, this is only a conjecture and needs to be proven or discarded by further investigations.

\begin{figure}[hbt!]
	
	\begin{subfigure}{0.5\textwidth}
		\centering
		\includegraphics[scale=0.4]{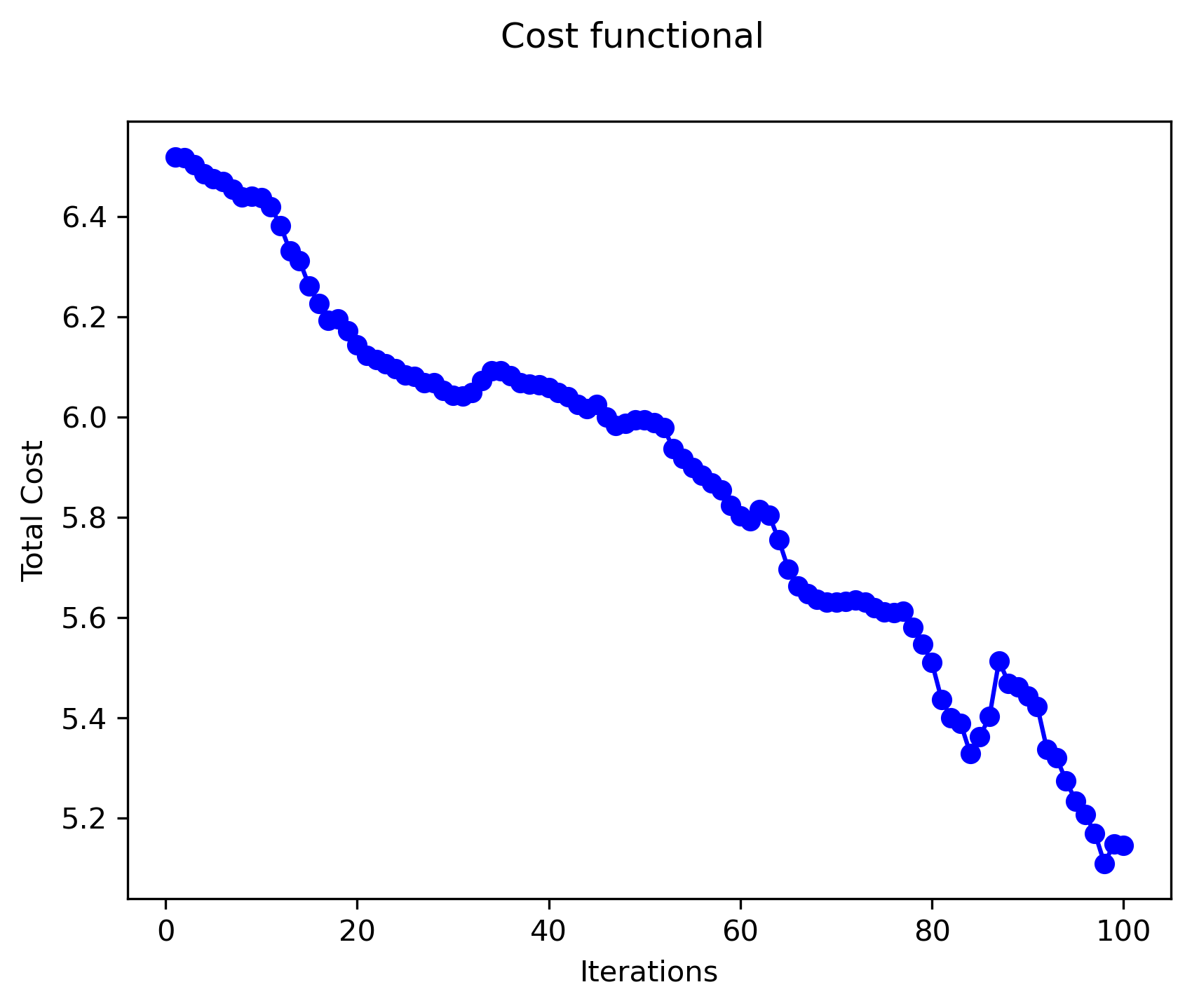}
		\caption{In the corridor}
		\label{fig:cost_corridor}
	\end{subfigure}
	\begin{subfigure}{0.5\textwidth}
		\centering
		\includegraphics[scale=0.422]{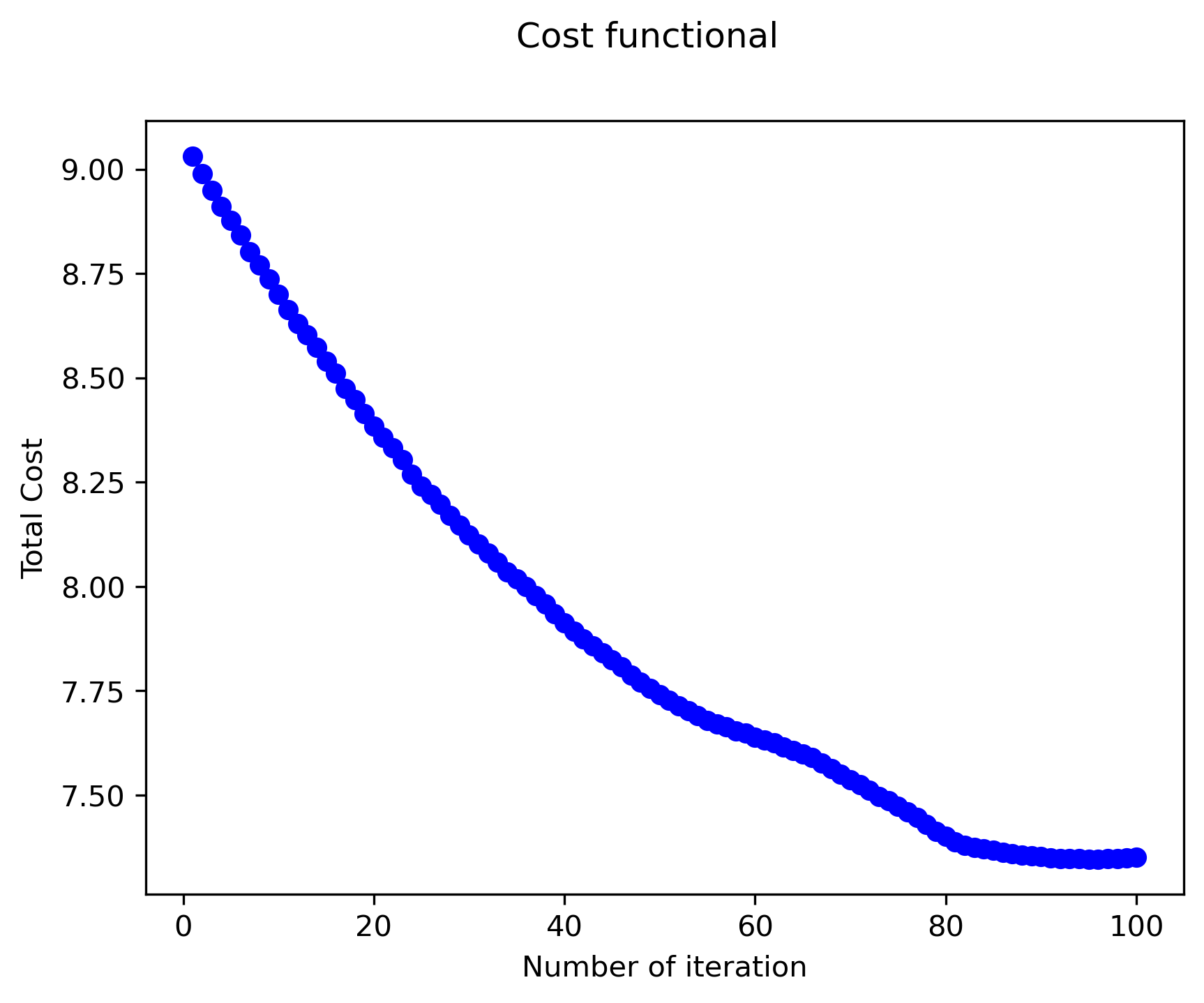}
		\caption{In the crossroad}
		\label{fig:cost_crossing}
	\end{subfigure}
	\caption{Cost functional}
	
\end{figure}

\section{Conclusion and outlook}
We extended the anisotropic interaction model proposed in \cite{Claudia} by including a body size and thus additional volume exclusion effects. Numerical studies indicate that the body size has influence on lane formation process. In fact, it seems that a smaller body size leads to a higher number of lanes formed and vice versa. Moreover, we investigated the fundamental diagram of the dynamics and found that higher densities lead to lower velocities. This observation was expected from other experiments and underlines the feasibility of the approach.

Due to the ODE formulation of the model, we were able to analyse the model in terms of well-posedness and further to rigorously derive a gradient-based descent algorithm for a calibration problem using real data. The optimal scaling parameters for the collision avoidance and the repulsion strength turn out to match very well for the two tested scenarios. For the attraction parameter we find different values for the two scenarios. We suspect that this is related to the post-collision behaviour of the model.

To prove or discard this conjecture is interesting future work. Moreover, the rigorous analysis of stationary states like the lanes in the corridor case and the travelling waves in the crossing case is planned for future investigations.

\begin{appendices}
	
	\section{Tensor  $d_{v_i}M^*(v_i,v_j)$ }
	\label{app:AppendixA}
	
	Here, we present the details of the mathematical manipulations made to tensor $d_{v_i}M(v_i,v_j)$ in the derivation of $d_{v_i}M^*(v_i,v_j)$.
	Note that $d_{v_i}M(v_i,v_j)$ is a tensor with dimension $2\times 2 \times 2 $, given by
	\begin{align}
		\label{eqn:M_derivative}
		d_{v_i}M\left( v_{i}, v_{j}\right) = \begin{pmatrix} 
			-\sin\alpha_{ij} \cdot d_{v_i}\alpha_{ij} & -\cos\alpha_{ij} \cdot d_{v_i}\alpha_{ij}\\
			\cos\alpha_{ij} \cdot d_{v_i}\alpha_{ij} & -\sin\alpha_{ij} \cdot d_{v_i}\alpha_{ij}
		\end{pmatrix}_{2\times 2 \times 2},
	\end{align}
	where
	\begin{align*}
		d_{ v_i} \alpha_{ij} =
		\begin{cases}
			-\lambda \cdot   \frac{1}{ \sqrt{(\norm{v_i}\norm{v_j})^2-\langle v_i,  v_j \rangle^2 }} 
			\cdot \left( v_j  - \langle v_i,  v_j \rangle  \frac{v_i}{\norm{v_i}^2}  \right) , & \text{for } v_i \neq 0, \  v_j \neq 0, \\
			\vec{0},                                                                      & \text{else}
		\end{cases}.
	\end{align*}
	The elements of the tensor $d_{v_i}M\left( v_{i}, v_{j}\right)$ in \eqref{eqn:M_derivative} we denote as 
	$$
	d_{v_i}M(v_i,v_j) = 
	\begin{pmatrix}
		\begin{pmatrix}
			m_{11} \\
			m_{12} 
		\end{pmatrix} 
		& \begin{pmatrix}
			m_{21} \\
			m_{22} 
		\end{pmatrix}  \\
		\begin{pmatrix}
			m_{31} \\
			m_{32} 
		\end{pmatrix} 
		& \begin{pmatrix}
			m_{41} \\
			m_{42} 
		\end{pmatrix}
	\end{pmatrix},
	$$ 
	then its dual $d_{v_i}M^*(v_i,v_j)$ in (\ref{eqn:E_total_v}) is the transposed tensor after swaping its' axes, as given below
	$$
	d_{v_i}M^*(v_i,v_j) = 
	\begin{pmatrix}
		\begin{pmatrix}
			m_{11} \\
			m_{21} 
		\end{pmatrix} 
		& \begin{pmatrix}
			m_{31} \\
			m_{41} 
		\end{pmatrix}  \\
		\begin{pmatrix}
			m_{12} \\
			m_{22} 
		\end{pmatrix} 
		& \begin{pmatrix}
			m_{32} \\
			m_{42} 
		\end{pmatrix}
	\end{pmatrix}.
	$$

	\section{Boundedness of $\nabla M(\cdot)$ }
	\label{app:AppendixB}
	
	We show the boundedness of $\nabla M(v_i, v_j )$, where
	
	\begin{equation*}
		\nabla M(v_i, v_j ) = \left(\frac{dM(v_i, v_j )}{d {v}_i}, \frac{dM(v_i, v_j )}{d {v}_j}\right),
	\end{equation*}
	with
	\begin{equation*}
		\label{eqn:dif_rotation_matrix1}
		\frac{dM}{d {v}_i}= \begin{pmatrix} 
			-\sin\alpha & -\cos\alpha \\
			\cos\alpha & -\sin\alpha 
		\end{pmatrix} \cdot \frac{d\alpha}{d {v}_i} , 
	\end{equation*}
	\begin{equation}
		\label{eqn:dif_alfa1}
		\frac{d\alpha}{d {v}_i} = 
		\begin{cases}
			-\lambda  \frac{1}{ \sqrt{(\norm{ {v_i} }\norm{{v_j} } ) ^2 - \langle {v_i}, {v_j} \rangle^2}}\left( {v_j} - \langle {v_i}, {v_j} \rangle \frac{{v_i}}{\norm{{v_i}}^2}\right) , & \text{if } {v_i}, {v_j}\neq 0 \\
			0,                                                                      & \text{else}
		\end{cases}  .
	\end{equation}
	
	We therefore look transform the system to polar coordinates and introduce the notations ${v_i = {\begin{pmatrix} 
				r_1 \cos \phi_1  \\
				r_1 \sin \phi_1  
			\end{pmatrix}, }}$
	and ${v_j = {\begin{pmatrix} 
				r_2 \cos \phi_2  \\
				r_2 \sin \phi_2  
	\end{pmatrix} } }$. Here $r_1$ and $r_2$ are positive scalars. 
	It is clear, that $\left|{\begin{pmatrix} 
			-\sin\alpha & -\cos\alpha \\
			\cos\alpha & -\sin\alpha 
	\end{pmatrix} }\right|  < \infty$. Then, we need to show $\left| \frac{d\alpha}{d {v}_i} \right| = C< \infty. $ 
	
	For further usage, we define  $${\norm{ v_i}_2  = r_1},\quad {\norm{ v_j}_2  = r_2} \text{ and }  { \langle {v_i}, {v_j} \rangle} = r_1 r_2 (\cos \phi_1 \cos \phi_2  + \sin \phi_1 \sin \phi_2). $$
	Substituting the velocity vectors in (\ref{eqn:dif_alfa1}), we get
	\begin{flalign*}
		\left|\frac{d\alpha}{d {v}_i} \right|_2 &= 
		\left| 
		\frac{-\lambda }{ \sqrt{r_1^2 r_2^2  - (r_1 r_2\cos (\phi_1 - \phi_2) )^2 }}
		\right| 
		\left| 
		\begin{pmatrix} 
			r_2 \cos \phi_2  \\
			r_2 \sin \phi_2  
		\end{pmatrix} - 
		r_1 r_2\cos (\phi_1 - \phi_2) \frac{1}{r_1^2}
		\begin{pmatrix} 
			r_1 \cos \phi_1  \\
			r_1 \sin \phi_1  
		\end{pmatrix}
		\right|_2 \\
		& =  
		\frac{\left| \lambda \right| }{ r_1 r_2 \left|\sin (\phi_1 - \phi_2)  \right| }
		\left| 
		\begin{pmatrix} 
			r_2 \cos \phi_2 - r_2 \cos \phi_1  \cos (\phi_1 - \phi_2) \\
			r_2 \sin \phi_2 - r_2 \sin \phi_1  \cos (\phi_1 - \phi_2) 
		\end{pmatrix} 
		\right|_2\\
		& =  
		\frac{\left| \lambda \right| }{ r_1 r_2 \left|\sin (\phi_1 - \phi_2)  \right| }
		[
		r_2^2 (\cos \phi_2 - \cos \phi_1  \cos (\phi_1 - \phi_2))^2 \\
		&+ r_2^2 ( \sin \phi_2 - \sin \phi_1  \cos (\phi_1 - \phi_2) )^2
		]^\frac{1}{2} \\
		&=  
		\frac{\left| \lambda \right| }{ r_1 r_2 \left|\sin (\phi_1 - \phi_2)  \right| }
		r_2 \sqrt{1 -  \cos^2 (\phi_1 - \phi_2)}  = 
		\frac{\left| \lambda \right| }{ r_1 } < \infty.
	\end{flalign*}
	
	Analogously, we can show that $\left|\frac{d\alpha}{d {v}_j} \right|_2 =\frac{\left| \lambda \right| }{ r_2 }< \infty$. This, proves the boundedness of $\nabla M(v_i, v_j ) $ as desired.

\section{Existence of adjoint states}\label{app:existenceAdjoint}
For completeness we give the assumptions and the statement of Corollary 1.3 of \cite{Rene_Pinnau} which we use to prove the existence of adjoint states in Section~\ref{ExistenceAdjoints}.

\begin{assumption}\label{ass:cor}
Let the following assumptions hold:
\begin{itemize}
	\item [(A1)] $\Uad \subset U$ is nonempty, convec and closed.
	\item [(A2)] $J \colon Y \times U \rightarrow \R$ and $e \colon Y \times U \rightarrow Z$ are continuously Fre\'echet differentialble and $U,Y,Z$ are Banach spaces.
	\item [(A3)] For all $u \in V$ in a neighborhood $V \subset U$ of $\Uad,$ the state equation $e(y,u) = 0$ has a unique solution $y = y(u) \in Y$.
	\item[(A4)] $e_y(y(u),u) \in \mathcal L(Y,Z)$ has a bounded inverse for all $u \in V.$
\end{itemize}
\end{assumption}

\begin{proposition}[Existence of adjoint states \cite{Rene_Pinnau}]\label{prop:adjoint}
	Let $(\bar y, \bar u)$ be an optimal solution of $$\min J(y,u) \text{ subject to } e(y,u)=0$$ and let Assumption~\ref{ass:cor} hold.
	
	Then there exists and adjoint state (or Lagrange multiplier) $\bar p \in Z^*$ such that the following optimality conditions hold
	\begin{align*}
	&\langle e(\bar y,\bar u), p \rangle_{Z,Z^*} = 0 &&\forall p \in Z^*, \\
	&\langle L_y(\bar y, \bar u, \bar p) , v \rangle_{Y^*,Y} = 0 &&\forall v \in Y^*, \\
	&\bar u \in \Uad, \langle L_u(\bar y, \bar u, \bar p) , u-\bar u \rangle_{U^*,U} \ge 0, &&\forall u\in \Uad.
	\end{align*}
	
\end{proposition}	
	
\end{appendices}

\section*{Acknowledgement}
ZT acknowledges funding by the German Academic Exchange Service programme "Mathematics in Industry and Commerce, MIC".

%\newpage

%%%%%%%%%%%%

\newpage

\end{document}